\documentclass{amsart}
\usepackage[latin1]{inputenc}
\usepackage[T1]{fontenc}
\usepackage[english]{babel}
\usepackage{csquotes}

\usepackage[margin=1.6in, top=1.4in, bottom=1.4in]{geometry}

\usepackage{color} 
\usepackage[colorlinks=true,linktoc=all,linkcolor=blue,citecolor=red,filecolor=pink,urlcolor=blue]{hyperref}
\usepackage[hyperref=true,style=numeric,backend=bibtex,maxnames=100,doi=false,isbn=false,url=false,giveninits=true]{biblatex}
\bibliography{biblio}

\usepackage{marvosym}
\usepackage{wasysym}

\usepackage{enumitem}
\usepackage{graphicx}
\graphicspath{{pictures/}}
\usepackage[justification=justified]{caption}

\usepackage[mathcal,mathscr]{euscript}

\usepackage{tikz}
\usetikzlibrary{arrows}

\usepackage{amssymb,amsmath,latexsym}
\theoremstyle{plain}                    
\newtheorem{theorem}{Theorem}[section]
\newtheorem{lemma}[theorem]{Lemma}
\newtheorem{proposition}[theorem]{Proposition}
\newtheorem{corollary}[theorem]{Corollary}
\theoremstyle{definition}
\newtheorem{definition}[theorem]{Definition}
\newtheorem{example}[theorem]{Example}
\theoremstyle{remark}
\newtheorem{remark}[theorem]{Remark}

\numberwithin{equation}{section}
\theoremstyle{plain}                    
\newtheorem{introtheorem}{Theorem}

\newcommand{\theoremnumber}{} 
\newtheorem*{maintheorem}{Theorem \theoremnumber}
\newenvironment{maintheoremc}[1]
  {\renewcommand{\theoremnumber}{#1}%
  \begin{maintheorem}}
  {\end{maintheorem}}

\newcommand{\mc}{\mathcal}
\newcommand{\Stab}{\mathrm{Stab}}

\newtheorem* {claim*}{Claim}

\newtheorem* {subclaim*}{Subclaim}

\newcommand{\leftQ}[2]{\left.\raisebox{-.2em}{$#2$}\middle\backslash\raisebox{.2em}{$#1$}\right.}
\newcommand{\doubleQ}[3]{\left.\raisebox{-.1em}{$#3$}\middle\backslash \raisebox{.1em}{$#1$}\middle/\raisebox{-.1em}{$#2$}\right.}

\newcommand{\co}{\colon\thinspace}
\newcommand{\CAT}{\ensuremath{\operatorname{CAT}}}

\newcommand{\zz}{\mathbb Z}

\newcommand{\rr}{\mathbb R}

\newcommand{\hh}{\mathbb H}

\newcommand{\ishyp}[1]{\operatorname{Isom}(\hh^{#1})}

\newcommand{\quathyp}[1]{\hh^{#1}_{\hh}}

\newcommand{\spisom}[1]{\operatorname{Isom}(\quathyp{#1})}
\newcommand{\slr}[1]{\operatorname{SL}(#1,\rr)}
\newcommand{\sor}[1]{\operatorname{SO}(#1,\rr)}
\newcommand{\spr}[1]{\operatorname{Sp}(1,#1)}

\newcommand{\stab}[2]{\operatorname{Stab}_{#1}({#2})}

\newcommand{\cat}[1]{\operatorname{CAT}(#1)}

\newcommand{\gromov}[1]{\mathcal G (#1)}

\newcommand{\hyperbolize}[1]{\mathcal H(#1)} 
\newcommand{\relhyp}[2]{\mathcal R(#1,#2)} 
\newcommand{\hypmt}[1]{\mathcal T(#1)} 

\newcommand{\cone}[1]{\operatorname C(#1)} 
\newcommand{\relcone}[2]{\operatorname C(#1,#2)} 

\newcommand{\hc}{X_\Gamma} 

\newcommand{\uchc}{\widetilde{\hc}} 
\newcommand{\buchc}{\widehat{\hc}} 
\newcommand{\icrhc}[2]{\mathcal R^\circ(#1,#2)} 
\newcommand{\ucrhc}[2]{\widetilde{\mathcal R}(#1,#2)} 

\newcommand{\hq}{\square^n_\Gamma} 
\newcommand{\uchq}{\widetilde{\hq}} 

\newcommand{\hg}{\Gamma_X} 

\newcommand{\cd}{g} 
\newcommand{\cdX}{\cd_X} 
\newcommand{\cdXbu}{\widehat \cdX} 

\newcommand{\foldbuchc}{\hat f} 

\newcommand{\piuchc}{\pi} 
\newcommand{\pibuchc}{\hat \pi} 

\newcommand{\dcc}[1]{\mathcal C(#1)} 
\newcommand{\dccx}{\mathcal C (\uchc)} 

\newcommand{\bv}{\mathcal{BV}} 
\newcommand{\cv}{\mathcal{CV}} 

\newcommand{\bmirrors}{\widehat {\mathcal M}} 
\newcommand{\mirrors}{\mathcal M} 
\newcommand{\bcomponents}{\widehat {\mathcal C}} 
\newcommand{\components}{\mathcal C} 
\newcommand{\bedgespace}{E_{M,C}^\varepsilon}
\newcommand{\edgespace}{E_{M,C}^\varepsilon}

\newcommand{\length}[1]{\ell(#1)}
\newcommand{\lk}[2]{\operatorname{lk}(#1,#2)}

\makeatletter
\@namedef{subjclassname@2020}{\textup{2020} Mathematics Subject Classification}
\makeatother

\begin{document}
\title[Relative cubulation of relative strict hyperbolization]{Relative cubulation of\\ relative strict hyperbolization}
\author{Jean-Fran\c{c}ois Lafont}
\address{Department of Mathematics - The Ohio State University, 231 West 18th Avenue, Columbus, OH 43210-1174, USA}
\email{jlafont@math.ohio-state.edu}
\author{Lorenzo Ruffoni}
\address{Department of Mathematics and Statistics - Binghamton University, Binghamton, NY 13902, USA}
\email{lorenzo.ruffoni2@gmail.com}
\date{\today}
 \subjclass[2020]{20F67, 53C23, 20E26, 57Q15}

\makeatletter
\let\@wraptoccontribs\wraptoccontribs
\makeatother
\contrib[with an appendix by]{Daniel Groves and Jason Manning}
\address{Department of Mathematics, Statistics, and Computer Science, University of Illinois at Chicago, 322 Science and Engineering Offices (M/C 249), 851 S. Morgan St., Chicago, IL 60607-7045}
\email{dgroves@uic.edu}

\address{Department of Mathematics, 310 Malott Hall, Cornell University, Ithaca, NY 14853}
\email{jfmanning@cornell.edu}

\begin{abstract}
We prove that many relatively hyperbolic groups obtained by  relative strict hyperbolization admit a cocompact action on a $\cat 0$ cubical complex.
Under suitable assumptions on the peripheral subgroups, these groups are residually finite and even virtually special.
We include some applications to the theory of manifolds, such as the construction of new non-positively curved Riemannian manifolds with residually finite fundamental group, and the existence of non-triangulable aspherical manifolds with virtually special fundamental group.
\end{abstract}

\maketitle

\tableofcontents


\section{Introduction}
\addtocontents{toc}{\protect\setcounter{tocdepth}{1}}

Hyperbolization procedures were introduced by Gromov in \cite{G87} as a way to construct aspherical manifolds which allows for more flexibility than the classical methods coming from Lie theory.
Roughly speaking, a hyperbolization procedure is defined by the choice of a \textit{hyperbolizing cell} and of a class of combinatorial complexes to be hyperbolized, and it consists in replacing the cells of a complex with copies of the chosen hyperbolizing cell.
The resulting space is often called a \textit{hyperbolized complex}, and it is a metric space of non-positive curvature in the sense of \cite{BH99}; in particular, it is aspherical.
In this paper we consider the \textit{strict} hyperbolization procedure introduced by Charney and Davis in \cite{CD95}, which can produce a space of strictly negative curvature.

A main feature of any hyperbolization procedure is that links of cells remain essentially unchanged under hyperbolization.
In particular, hyperbolizing a triangulation of a closed (pseudo-)manifold results in a closed aspherical (pseudo-)manifold.
This has been used to construct examples of closed aspherical (pseudo-)manifolds that display various pathological features; see \cite{DJ91,CD95,O20} for a few examples.
For some applications, one may want to preserve a certain subcomplex of the original complex, i.e., ensure that the hyperbolized complex contains a certain prescribed subspace. 
This can be arranged via \textit{relative} hyperbolization procedures, which consist in coning off the desired subcomplex, performing a non-relative hyperbolization, and then removing a neighborhood of the cone point. 
The link of the cone point is homeomorphic to the desired subcomplex.
See \cite{DJW01,BE06,BZH06,NP13,DFL14,SW20,RU23,AOS24,LW24} for a few examples of this strategy.

Since the spaces obtained from hyperbolization are aspherical, a natural problem is to understand what their fundamental groups look like.
As one might expect, the fundamental group of a space obtained via Charney-Davis strict hyperbolization is a hyperbolic group.
In \cite{LR24}  some mild conditions have been identified under which these hyperbolic groups are virtually special.
In particular, all the closed aspherical manifolds obtained via Charney-Davis strict hyperbolization have linear, hence residually finite, fundamental group.

Similarly, the fundamental group of a space obtained via the relative version of Charney-Davis strict hyperbolization is a relatively hyperbolic group, with peripheral structure given by the fundamental group of the desired subcomplex; see \cite{BE07}.
In this paper we aim at extending the results of \cite{LR24} to the relative setting.
To this end, we consider the relatively hyperbolic groups that arise from relative strict hyperbolization and construct certain actions on $\cat 0$ cubical complexes.
Let us fix some objects and notations that appear in the statement of our main results.

Let $K$ be a finite simplicial complex and $L$ a subcomplex.
Assume that the complex $\relcone KL$ obtained by coning off each component of $L$ is homogeneous and without boundary (i.e., each simplex is contained in an $n$-simplex, and each $(n-1)$-simplex is contained in at least two $n$-simplices).
Note that this implies in particular that $L$ is homogeneous of dimension $n-1$ and without boundary.
As a motivating example, consider the case in which $K$ is a triangulation of a compact manifold with boundary and $L$ is the induced triangulation of the boundary.
This is a common setup in many contexts in which relative hyperbolization procedures are used, and will be a standing assumption throughout this paper.
Let $\relhyp KL$ denote the relative strict hyperbolization of $K$ with respect to $L$; see \cite{DJW01,BE07} or \S\ref{sec:rel hyp} below for details.
There is a natural $\pi_1$-injective embedding of $L$ into $\relhyp KL$.
Let $\mathcal P$ be a set of representatives of the conjugacy classes of the fundamental groups of the components of $L$.
Then the group $G=\pi_1(\relhyp KL)$ is known to be relatively hyperbolic with respect to $\mathcal P$.
Our main result is the following.

\begin{introtheorem}\label{thm:intro main}
The group $G=\pi_1(\relhyp KL)$ acts on a $\cat 0$ cubical complex $\dcc{\buchc}$ by isometries and satisfying  the following properties.
\begin{enumerate}
    \item  $\leftQ{\dcc \buchc}{G}$ is compact.
    \item Each $P \in \mc{P}$ acts elliptically on $\dcc \buchc$.
    
    \item For each cube $C$ of $\dcc \buchc$, $\stab{G}{C}$ is either maximal parabolic, or else is full relatively quasi-convex, hyperbolic, and virtually compact special. 
\end{enumerate}
\end{introtheorem}

This set up does not automatically fit in the  available literature about cubulation of relatively hyperbolic groups, such as \cite{EG20,EG22,GM22,RE23}.
Indeed, the action of $G$ on $\dcc{\buchc}$ is very far from proper, because most vertices have infinite stabilizer.
This action is not relatively geometric in the sense of \cite{EG20,EG22}, and not even weakly relatively geometric in the sense of  \cite{GM22}.
The problem here does not arise from the parabolics, but rather from points that are away from the fixed points of the parabolics.
It is the same lack of properness already encountered in \cite{LR24}, which required a criterion obtained by Groves and Manning in \cite{GM23} to promote an improper action of a hyperbolic group to a proper one (on a different cubical complex).
A relative analogue of \cite{GM23} is not available yet in full generality. 
However, one can obtain the following.

\begin{introtheorem}\label{introthm: res fin special main}
Let $G=\pi_1(\relhyp KL)$ and $\mathcal P$ be as above.
Then the following hold.
\begin{enumerate}
    \item \label{item GM app} If each $P\in \mathcal P$ is residually finite, then $G$ is residually finite and each $P\in \mathcal P$ is separable.

    \item If each $P\in \mathcal P$ is hyperbolic and virtually compact special, then $G$ is hyperbolic and virtually compact special.
\end{enumerate}
\end{introtheorem} 
Here \eqref{item GM app} follows from a more general result contained in the Appendix~\ref{sec:appendix} by Groves and Manning (see Theorem~\ref{thm: appendix main}).
 In particular, Theorem~\ref{introthm: res fin special main} shows that relative strict hyperbolization is unlikely to provide a negative answer to \cite[Problem 6.6]{OS06}, which is the relative analogue of the well-known question of Gromov about residual finiteness of hyperbolic groups.
 A  result analogous to Theorem~\ref{introthm: res fin special main} has recently been obtained by Avramidi, Okun, and Schreve  for a different relative hyperbolization procedure, obtained by combining the Davis reflection group trick with the Charney-Davis strict hyperbolization; see  \cite[Theorem D]{AOS24}.

 \medskip
We conclude the introduction with a brief description of some applications to the theory of manifolds.

\subsubsection*{Aspherical manifolds with residually finite fundamental groups}
 In \S\ref{sec:new examples} we obtain examples of closed aspherical manifolds which have residually finite fundamental group.
    This is based on a construction that we call \textit{hyperbolized mapping torus}, and which has already been considered for instance in \cite{NP13,RU23}.
    It consists in taking a manifold $M$, hyperbolizing $M\times [-1,1]$ relatively to the boundary, and then gluing the two boundary components to get a closed manifold $\hypmt M$.
    We show that certain properties of the fundamental group of $M$ (such as residual finiteness) are inherited by the fundamental group of $\hypmt M$.
    
    By choosing the  manifold $M$ appropriately, one can thus obtain examples which are ``new'', in the sense that they are not homotopy equivalent to manifolds for which residual finiteness of the fundamental group was previously known: 
    see Theorem~\ref{thm: new main thm} (non-positively curved, dimension $n \geq 6$) and Remark~\ref{rmk: new neg curved} (negatively curved, dimension $n \geq 9$).
    These examples are obtained starting  from lattices in $\slr 3$ and in $ \spr 2 = \spisom 2$ respectively.
    
    For $n\geq 5$, we can also use this construction to obtain some closed Riemannian manifolds having negative curvature and virtually compact special fundamental group; see Theorem~\ref{thm: new special}. 
    These can be obtained starting from Gromov-Thurston manifolds or strictly hyperbolized manifolds.
    On the other hand, it is worth mentioning that there are closed aspherical manifolds that do not virtually fiber over the circle and whose fundamental group is hyperbolic and virtually compact special; see \cite[Theorem A]{AOS24} for odd-dimensional examples.

\subsubsection*{Cobordism of manifolds}
In \S\ref{sec: cobordism} we consider some classical applications of hyperbolization procedures to cobordism of manifolds.
    In Corollary~\ref{cor: cobordism} we obtain that a cobordism between triangulable manifolds with residually finite fundamental group can be chosen to have residually finite fundamental group.
    Similarly, in Corollary~\ref{cor: flat} we obtain that every closed flat manifold bounds geometrically a manifold of pinched negative curvature with residually finite  fundamental group.
    In both cases, the ambient fundamental group is relatively hyperbolic and the fundamental groups of the boundary components are separable.

\subsubsection*{Aspherical manifolds that cannot be triangulated}
Finally, in \S\ref{sec: triangulability} we show that for $n\geq 6$ there is a closed aspherical manifold that cannot be triangulated and whose fundamental group is hyperbolic and virtually compact special; see Theorem~\ref{thm: new no tri}.
These are the manifolds constructed in \cite{DFL14} by gluing together two pieces, one obtained via strict hyperbolization and one obtained via relative strict hyperbolization.
Combining the results of \cite{LR24} and of this paper, we show that these manifolds have virtually special fundamental group.
In particular, they have a rich collection of finite covers, and it makes sense to ask if these manifolds are \textit{virtually triangulable}, i.e., if they admit a finite cover that can be triangulated.
While we do not answer this question, 
we show that no cover of odd degree can be triangulated.
The situation for  covers of even degree is more delicate. For instance, the Galewski-Stern $5$-manifold from \cite{GS79} is non-triangulable and non-orientable, but its orientable double cover is triangulable, since all orientable closed 5-manifolds are triangulable by \cite{SI70}.

\subsection*{Outline}
In \S\ref{sec:preliminaries} we collect background notation and terminology.
In \S\ref{sec:rel hyp} we present the relative strict hyperbolization procedure, introduce the spaces of interest in this paper, and present some lemmas about the action of the relatively hyperbolized groups on them.
In \S\ref{sec:dual complex} we construct the dual cubical complex, prove it is $\cat 0$, and complete the analysis of stabilizers. 
The proofs of Theorem~\ref{thm:intro main} and Theorem~\ref{introthm: res fin special main} appear in \S\ref{sec:proofs}.
In \S\ref{sec:applications} we discuss some applications to the theory of manifolds.
The Appendix \ref{sec:appendix} by Groves and Manning contains the proof that a group satisfying the conclusion of Theorem~\ref{thm:intro main} with respect to residually finite peripherals has separable full relatively quasiconvex subgroups.

\subsection*{Acknowledgements}
We would like to thank Igor Belegradek and Mike Davis for useful conversations, Corey Bregman and Alain Valette for their comments, and the referees for their careful reading and helpful remarks.
J.-F. Lafont was partly supported by the NSF Grant number DMS-2109683 and DMS-2407438. 
L. Ruffoni acknowledges partial support by INDAM-GNSAGA, the AMS and the Simons Foundation.


\section{Preliminaries}\label{sec:preliminaries}
\addtocontents{toc}{\protect\setcounter{tocdepth}{1}}
In this section we fix review some standard background and  terminology needed in this paper.

\subsection{Cell complexes}
For background on cell complexes we refer the reader to \cite{BH99}. 
A \textit{cell complex} is a topological space $X$ obtained by gluing together cells along their faces, in such a way that each cell embeds in $X$ and the intersection of any two cells is either empty or a cell.
A \textit{simplicial complex} is a cell complex obtained by gluing copies of the standard simplex $\triangle^n$.
A \textit{cubical complex} is a cell complex obtained by gluing copies of the standard cube $\square^n=[0,1]^n$. 
The \textit{dimension} of a cell complex is the maximum dimension of its cells. 
We say that an $n$-dimensional cell complex is \textit{homogeneous} if every cell is contained in a cell of dimension $n$, and that it is \textit{without boundary} if every $(n-1)$-cell is contained in at least two different $n$-cells.

A cell complex $X$ is \textit{piecewise} \textit{spherical}, \textit{Euclidean}, or \textit{hyperbolic} if its cells can be realized as totally geodesic cells in a round sphere $\mathbb S^n$, a Euclidean space $\mathbb E^n$, or the (real) hyperbolic space $\hh^n$, and the gluing maps can be realized by isometries.
If $X$ has finitely many isometry classes of cells then the metrics defined on the cells can be glued together and $X$ is a complete geodesic metric space (see \cite[Theorem I.7.50]{BH99}). 
In particular, if $X$ is simplicial or cubical then it carries a standard piecewise Euclidean metric. 
 
If $X$ and $Y$ are cell complexes, a continuous function $f:X\to Y$ is a \textit{combinatorial map} if for every cell $C$ of $X$ we have that $f$ is a homeomorphism from $C$ to a cell $f(C)$ of $Y$.
A simplicial $n$-dimensional complex $X$ is \textit{foldable} if it admits a combinatorial map $f:X\to \triangle^n$ which is injective on  each simplex.
Such a map will be called a \textit{folding} for $X$.
An analogous definition can be given in the cubical case in terms of a map  to $\square^n$.
The barycentric subdivision of any cell complex is a foldable simplicial complex.
Also note that the cells of a foldable complex  are necessarily embedded.

 The \textit{link} of a point $p$  in a cell complex $X$ is defined to be the space  of unit vectors at $p$ that point into the cells of $X$ that contain $p$, and is denoted $\lk pX$.
 Similarly, the \textit{link} of a cell $C$ in a cell complex $X$ is defined as the space of unit vectors  based at an interior point of $C$ which are orthogonal to $C$ and point into the cells that contain $C$. It is denoted $\lk CX$.
 Observe that links of points or cells are naturally piecewise spherical complexes, and that if $X$ is simplicial or cubical, then all links are simplicial.

Finally, note that the \textit{cone} $\cone X$ over a cell complex $X$ admits two natural topologies, namely the cone topology (i.e., the quotient topology coming from the quotient map $X\times [0,1]\to \cone X$ that defines the cone $\cone X$), and the metric topology (i.e., the one coming from the fact that $\cone X$ has a natural structure of a cell complex.)
If $X$ is finite, then the two topologies agree, but when $X$ is infinite the cone topology is much finer; in particular it is not first countable, hence not metrizable. 
Similarly, if $X$ is a cell complex and $Y$ is a subcomplex, then the \textit{relative cone} $\relcone XY$ obtained by attaching the cone over $Y$ to $X$ is endowed with the metric topology. 
We also say that $\relcone XY$ is obtained from $X$ by \textit{coning off} $Y$.


\subsection{Bounded curvature}
We will consider the usual notions of curvature for metric spaces, see \cite[\S II.1, \S III.H.1]{BH99} for details.
A space is \textit{non-positively curved} if it is locally $\cat 0$, and \textit{negatively curved} if it is locally $\cat k$ for some $k<0$. 
In this paper ``hyperbolic'' always means ``Gromov hyperbolic'' unless otherwise specified.
A $\cat 0$ space is contractible; a $\cat{k}$ space is hyperbolic as soon as $k<0$.
The action of a group on a metric space is \textit{geometric} if it is cocompact, proper, and isometric.
A group is $\cat 0$ (resp. hyperbolic) if it admits a geometric action on a $\cat 0$ (resp. hyperbolic) space.

We recall the following well-known characterization of non-positive curvature for cubical complexes, see \cite[Theorems II.5.20]{BH99}).
Recall that links in a cubical complex are simplicial.
A simplicial complex is \textit{flag} if any $k+1$ pairwise adjacent vertices span a $k$--simplex.

\begin{lemma}[Gromov's link condition]\label{lem:gromov link condition}
Let $X$ be a cubical complex. Then $X$ is non--positively curved if and only if the link of each vertex is flag. 
\end{lemma}


Finally, let us recall some standard terminology about relative hyperbolicity; see \cite{HR10} for details.
Let $G$ be a finitely generated group and let $\mathcal P$ be a finite collection of subgroups of $G$.
Given a Cayley graph $\Gamma$ for $G$, we can construct the \textit{coned-off Cayley graph} $\widehat \Gamma$ by adding a vertex $v_{gP}$ for each coset of a subgroup $P\in \mathcal P$ and attaching $v_{gP}$ to vertices in $gP$ with edges of length $\frac 12$.
We say $(G,\mathcal P)$ is \textit{relatively hyperbolic} if some (every) coned-off Cayley graph $\widehat \Gamma$ is hyperbolic and has the Bounded Coset Penetration property.
We will also say that $G$ is \textit{hyperbolic relative to} $\mathcal P$; see \cite{HR10} for details and other equivalent definitions.
A subgroup  $P\in \mathcal P$ is a \textit{peripheral subgroup} and $\mathcal P$ is a \textit{peripheral structure} on $G$.
A conjugate of a peripheral subgroup is a \textit{maximal parabolic subgroup}, and a \textit{parabolic subgroup} is a subgroup of a maximal parabolic subgroup.
Finally, we say that a subgroup $H\subseteq G$ is \textit{relatively quasiconvex} if it is quasiconvex in $\widehat \Gamma$, i.e., if there is constant $K\geq 0$ such that every geodesic in $\widehat \Gamma$ with endpoints in $H$ lies at distance at most $K$ from $H$.
(In our setting, this is equivalent to being quasiconvex in the \textit{relative Cayley graph}, obtained by adding $\mathcal P$ to the generating set. Indeed, these two graphs are quasi-isometric and induce the same metric on $G$, seen as a subset of the vertex set in each case.)

The following criterion for relative quasiconvexity is well-known to experts and follows from a relative version of Milnor--\v Svarc obtained by Charney and Crisp in \cite{CC07}.
We include a proof for the convenience of the reader.
Here we say an action is \textit{discontinuous} if orbits are discrete and an \textit{isotropy subgroup} is a subgroup with non-empty fixed set.

\begin{lemma}\label{lem:relative MilnorSvarc}
Let $X$ be a hyperbolic length space.
Let $G$ be a finitely generated group admitting a discontinuous, cocompact, isometric action on $X$.
Let $\mathcal P$ be a collection of subgroups of $G$ consisting of exactly one representative from each conjugacy class of the maximal isotropy subgroups for this action.
Assume that $G$ is relatively hyperbolic with respect to $\mathcal P$.
Let $Y\subseteq X$ be a quasiconvex subset. 
Let $H$ be the stabilizer of $Y$ in $G$. 
If $H$ acts cocompactly on $Y$, then $H$ is a relatively quasiconvex subgroup of $(G,\mathcal P)$.  
\end{lemma}
\proof
Fix a basepoint $p\in Y$.
It follows from the proof of \cite[Theorem 5.1]{CC07} that the orbit map
$G\to X, g\mapsto g.p$
is a quasi-isometry, when $G$ is equipped with the metric induced from $ \widehat{\Gamma}$.
Since every point of $ \widehat{\Gamma}$ is at distance at most $\frac 12$ from $G$, this  orbit map extends to  a $(\lambda,\varepsilon)$-quasi-isometry
$f: \widehat \Gamma \to X$
for some constants $\lambda \geq 1, \varepsilon>0$.
Now the proof proceeds as in the absolute case.

Let $\delta$ be the hyperbolicity constant of $X$ and $K$ the quasiconvexity constant of $Y$.
Now pick $h\in H$ and a geodesic $\gamma:I\to \widehat \Gamma $ from $1_G $ to $h$. 
Then $f(\gamma)$  is a $(\lambda,\varepsilon)$-quasi-geodesic in $X$ from $p$ to $f(h)=h.p$.
Let $\alpha$ be a geodesic path in $X$ from $p$ to $f(h)=h.p$.
By the Morse Lemma (see for instance \cite[Theorem III.H.1.7]{BH99}) we get that $f(\gamma) \subseteq \mathcal N_A(\alpha)$ for some constant $A=A(\delta,\lambda, \varepsilon)$. Moreover since $p,h.p\in Y$ and $Y$ is $K$-quasiconvex, we get also that $\alpha\subseteq \mathcal N_K(Y)$. 
Finally since $H$ acts cocompactly on $Y$, there is some $F>0$ such that $Y \subseteq \mathcal N_F(f(H))$.
Combining these statements we see that $f(\gamma) \subseteq \mathcal N_B(f(H))$ for some $B=B(\delta,K,\lambda, \varepsilon,F)$.

In particular, for each $x\in \gamma$ there exists some $h_x\in H$ such that $d_X(f(x),f(h_x))\leq B$.
But then, since $f$ is a $(\lambda,\varepsilon)$-quasi-isometric embedding, we also get
$$d_{\widehat \Gamma}(x,h_x)\leq \lambda ( d_X(f(x),f(h_x)) +\varepsilon)\leq \lambda (B+\varepsilon)$$
which shows that there exists a constant $C=C(\delta,K,\lambda, \varepsilon,F)$ such that $\gamma \subseteq \mathcal N_C(H)$, i.e. $H$ is $C$-quasiconvex in $\widehat \Gamma$.
\endproof

In particular, when $Y$ is just the fixed point of a maximal parabolic subgroup one recovers the fact that each maximal parabolic subgroup is relatively quasiconvex.

Finally, we will need the following definitions.
A subgroup $H \subseteq G$  is \textit{full relatively quasiconvex} if it is relatively quasiconvex and for any $g\in G$ and for any peripheral subgroup $P\in \mathcal P$ the intersection $H \cap g P g^{-1}$ is either finite or of finite index in $g P g^{-1}$.
For instance, since the collection of maximal parabolic subgroups is almost malnormal,  a maximal parabolic subgroup is automatically a full relatively quasiconvex subgroup.
A subgroup $H \subseteq G$  is \textit{strongly relatively quasiconvex} if it is relatively quasiconvex and for any $g\in G$ and for any peripheral subgroup $P\in \mathcal P$ the intersection $H \cap g P g^{-1}$ is finite.



\section{Relative strict hyperbolization}\label{sec:rel hyp}
\addtocontents{toc}{\protect\setcounter{tocdepth}{2}}
Let $K$ be a finite connected simplicial complex, and $L$ a (not necessarily connected) subcomplex.
The main example the reader should keep in mind is when $K$ is a compact manifold with boundary $L$.

\subsection{The basic spaces}\label{sec: basic spaces}
We consider the relative hyperbolization procedure of \cite[\S 2]{DJW01}, as well as its strict version, see \cite{BE07}. See Figure~\ref{fig:rel hyp} for a picture of the main steps involved in the construction.
For an exposition of Gromov's cylinder construction and Charney-Davis strict hyperbolization we refer the reader to \cite{LR24} and references therein.

\begin{figure}[h]
\centering
\def\svgwidth{.7\columnwidth}
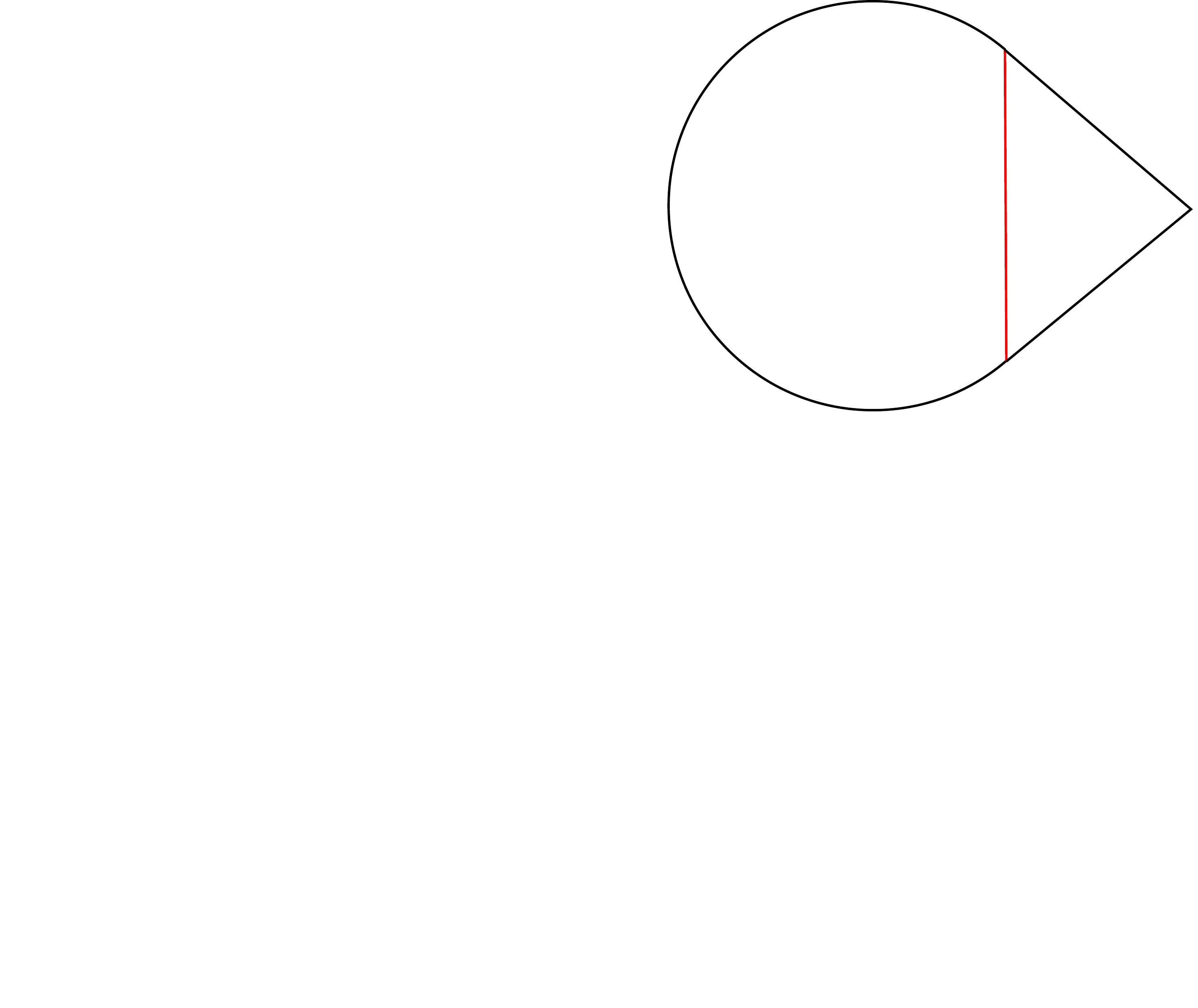
    \caption{The relative strict hyperbolization procedure.}
    \label{fig:rel hyp}
\end{figure}

    Let $\relcone KL$ be the simplicial complex obtained by attaching a cone $C_i$ over each connected component $L_i$ of $L$. 
    Denote by $y_i$ the cone point.
    We will  assume that $\relcone KL$ is foldable, which can always be achieved by taking a barycentric subdivision of $\relcone KL$.
    
    Let $X = \gromov {\relcone KL}$ be the cubical complex obtained by applying Gromov's cylinder construction. 
    Then $X$ is a non-positively curved and foldable cubical complex.
    Cone points arising from the cone points of $\relcone KL$ are still denoted $y_i$.
    
    Let $\hc = \hyperbolize{X}= \hyperbolize{\gromov {\relcone KL}}$ be the piecewise hyperbolic complex obtained by applying the Charney-Davis strict hyperbolization from \cite{CD95} to $X = \gromov {\relcone KL}$. 
    Roughly speaking, this is obtained by replacing each cube of $X$ with a certain hyperbolic manifold with corners  $\hq$.
    We still denote by $y_i$ the points arising from the cone points.
    Links are essentially preserved, in the sense that $\lk{y_i}{\hc}$ is isomorphic to a subdivision of $\lk{y_i}{\relcone KL}=L_i$.

    Let $\relhyp KL$ be the space obtained from $\hc$ by removing a small open  ball around each $y_i$.
    This is the \textit{relative strict  hyperbolization} of $K$ with respect to $L$.
    Note that $L$ embeds (up to subdivision) as a  subcomplex of $\relhyp KL$.
    The following are the  main features of this procedure, see \cite{DJW01,BE07}.
    \begin{enumerate}
        \item $\relhyp KL$ is aspherical if and only if each  component of $L$ is aspherical.
    
        \item If $K$ is a PL manifold with boundary $L$, then $\relhyp KL$ is a PL manifold with boundary $L$.
        The same is true in the smooth category, if one works with smooth triangulations.

        \item Each component $L_i$ is $\pi_1$-injective.
        If $\mathcal P$ denotes a set of representatives of the conjugacy classes of the subgroups $\pi_1(L_i)$, then Belegradek proved in  \cite{BE07} that $G=\pi_1(\relhyp KL)$ is   hyperbolic relative to  $\mathcal P$.
        Note that since $K$ is finite, $\mathcal P$ is finite and $P\in \mathcal P$ is finitely presented.
    \end{enumerate}

\begin{remark}
$\hc$ is negatively curved, but in general the metric induced on $\relhyp KL$ is not even non-positively curved.
\end{remark}

\begin{remark}
 If $K$ is compact and homogeneous, and $L=\partial K$, then $\pi_1(\hc)$ is hyperbolic and virtually compact special by \cite{LR24}. However, $\relhyp KL$ is not $\pi_1$-injective in $\hc$, so $G=\pi_1(\relhyp KL)$ is not naturally a subgroup of $\pi_1(\hc)$.
\end{remark}

\begin{remark}[A standing assumption]\label{rem:assumption}
To use the techniques from \cite{LR24} we need one additional assumption, namely that the pair $(K,L)$ is such that $\relcone KL$ is homogeneous and without boundary.
This condition ensures that $\gromov{\relcone KL}$ is an \textit{admissible} cubical complex in the sense of \cite[\S 3]{LR24}.
We will assume this condition in this paper.
In particular, $X=\gromov{\relcone KL}$ comes with a folding $f:X\to \square^n$.
\end{remark}

\begin{example}
The motivating case is when $K$ is a compact manifold with boundary $L$.
More generally, $K$ could be any homogeneous complex with boundary $L=\partial K$, or we could take $K$ to be homogeneous and without boundary and $L$ to be a homogeneous codimension-1 subcomplex without boundary.
\end{example}

\subsection{Some useful covering spaces}\label{sec:auxiliary covering spaces}
Our goal is to construct a nice action of $G=\pi_1(\relhyp KL)$ on a $\cat 0$ cubical complex.
We follow \cite{LR24} and construct a  cubical complex dual to a certain collection of subspaces (mirrors) defined on a suitable cover of $\hc$.
Some adaptations are needed to work relatively to $L$.
The idea is to follow the treatment in \cite[\S 2]{DJW01}; see Figure~\ref{fig:covers}.

\begin{figure}[h]
\centering
\def\svgwidth{\columnwidth}
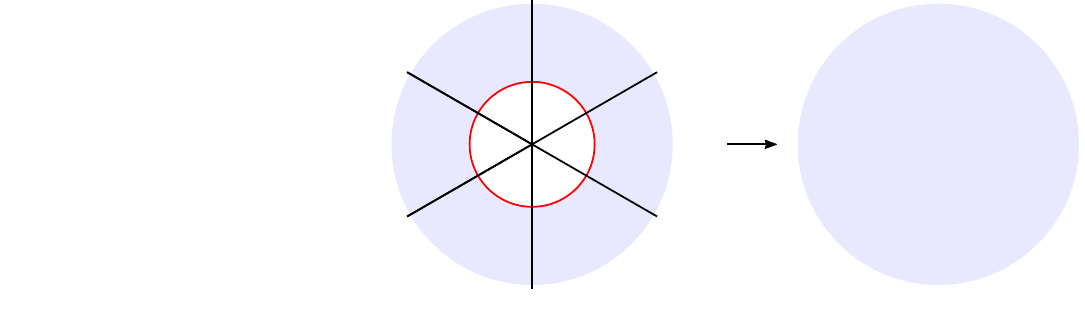
    \caption{Some covering spaces, local pictures around a cone/branch point.}    
    \label{fig:covers}
\end{figure}

    Let $\uchc$ be the universal cover of $\hc$.
    Note that the link of each lift of a cone point $y_i$ is naturally identified with a subdivision of $L_i$.
    
    Let $\icrhc KL$ be the lift of $\relhyp KL$ to $\uchc$.
    This is a covering space of $\relhyp KL$, but it is in general not simply connected.
    Indeed, there is an injection $\pi_1(L_i)\to \pi_1(\icrhc KL)$, because there is a retraction $\icrhc KL \to \lk{y_i}{\uchc} \cong L_i$ for each cone point $y_i$ 
    (see \cite[Lemma 2.2]{DJW01}).

    Let $\buchc$ be the \textit{branched universal cover} of $\hc$, i.e the space obtained by puncturing $\uchc$ at all the cone points, taking the universal cover, and then taking the metric completion.
    The ideal points in the completion are isolated points, covering the cone points, and we call them \textit{branch points}.
    The space $\buchc$ is a piecewise hyperbolic complex, which is locally isometric to $\uchc$, except at the branch points. 
    The link of a branch point above a cone point $y_i$ is isomorphic to a subdivision of the universal cover $\widetilde L_i$ of $L_i$.
    Since $\uchc$ is $\cat{-1}$, we have that $L_i$ is $\cat 1$, and therefore $\widetilde L_i$ is $\cat 1$. It follows that $\buchc$ is a    $\cat{-1}$ piecewise hyperbolic complex; see \cite[\S II.5]{BH99} for details.

     Finally, let $\ucrhc KL$ be the lift of $\icrhc KL$ to $\buchc$.
     Note that $\widetilde L_i$ is simply connected and each branch point has a simply connected neighborhood.
     By Seifert--Van Kampen we get that $\ucrhc KL$ is simply connected, so $\ucrhc KL$ is the universal cover of $\relhyp KL$.

\begin{remark}\label{rem:cover is coned}
The space $\buchc$ is homeomorphic to the space obtained from the universal cover $\ucrhc KL$ by coning off each copy of the universal cover $\widetilde L_i$ of $L_i$.
Moreover, there is a natural continuous map $\pibuchc: \buchc \to \uchc$. This is a covering map in the complement of the branch points.
\end{remark}

 
As a result of the above construction, we obtain a \textit{folding map}

$$
    \foldbuchc: \buchc \overset{\pibuchc}{\to} \uchc \overset{\pi}{\to} \hc \overset{\cdX}{\to} X \overset{f}{\to} \square^n.
$$
As   in \cite{LR24}, we define a \textit{mirror} in $\buchc$, $\uchc$ and $\hc$ to be a connected component of the full preimage of a codimension-1 face of $\square^n$ via this folding map.
The collection of mirrors induces a stratification of $\buchc$, and we can define a $k$-\textit{cell} to be the closure of a connected component of the full preimage of an open $k$-face of $\square^n$.
In particular, a $0$-cell  is a vertex, a $1$-cell is an edge, and an $n$-cell  is called a \textit{tile}, and is isometric to the universal cover $\uchq$ of $\hq$ (recall this is the hyperbolic manifold with corners constructed in \cite{CD95} to define the strict hyperbolization procedure).
This provides a \textit{stratification} and a \textit{tiling} of $\buchc$, analogous to the ones obtained for $\uchc$ in \cite{LR24}.

Any mirror going through a branch point might have some non locally finite behavior, but the collection of mirrors itself is locally finite: at most $n=\dim X$ mirrors go through each point of $\buchc$.


\begin{remark}[The case of $L$ with simply connected components]\label{rem:easy case}
 The link of a cone point of $\hc$ is  homeomorphic to the corresponding component of $L$.
In particular, if each component  of $L$ is simply connected, then each cone point has a simply connected link and a simply connected  neighborhood (namely, the cone over the link).
It follows from the Seifert-Van Kampen theorem that puncturing $\hc$ at the (finitely many) cone points does not change the fundamental group, i.e., the inclusion $\relhyp KL \subseteq \hc$  induces an isomorphism on fundamental groups. 
In particular, the group $G=\pi_1(\relhyp KL) = \pi_1(\hc)$ is hyperbolic and virtually compact special by \cite{LR24}.

For a motivating example, let $K$ be a manifold whose boundary $L=\partial K$ has simply connected components.
In \cite{RU23} this set up was considered to construct manifolds with hyperbolic fundamental group that do not admit any real projective or flat conformal structures, in any dimension at least $5$.
\end{remark}


\subsection{The action of \texorpdfstring{$G$}{G} on \texorpdfstring{$\buchc$}{the branched cover}}
Recall from \S\ref{sec: basic spaces} that if $\mathcal P$ is a set of representatives of the conjugacy classes of the subgroups $\pi_1(L_i)$, where $L_i$ is a connected component of $L$, then the group $G=\pi_1(\relhyp KL)$ is hyperbolic relative to $\mathcal P$.
Moreover, $G$ naturally acts on $\ucrhc KL$ by deck transformations.
As noted in Remark~\ref{rem:cover is coned} this space embeds in $\buchc$, the complement being given by the cones over the various copies of the universal cover $\widetilde L_i$ of $L_i$.
Note that the action permutes these subspaces  $\widetilde L_i$, and each maximal parabolic subgroup stabilizes one of them.

We extend the action of $G$ to the entire $\buchc$ by defining it on the cones in the obvious way.
The resulting action is cocompact and by isometries, but not proper, because each maximal parabolic fixes a branch point.
To address the failure of properness, we now study the cell stabilizers for the action of $G$ on $\buchc$.

\begin{lemma}\label{lem:stab cell}
Let $\sigma$ be a cell of $\buchc$.
\begin{enumerate}
    \item \label{item:stab cell branched} If $\sigma$ is  a branch point, then $\stab{G}{\sigma}$ is a  maximal parabolic subgroup.

    \item \label{item:stab cell unbranched} If $\sigma$ is not a branch point, then $\stab{G}{\sigma}$ is a hyperbolic and virtually compact special group.
\end{enumerate}

\end{lemma}
\begin{proof}
Since \eqref{item:stab cell branched} follows directly from the construction of the space $\buchc$, we just prove \eqref{item:stab cell unbranched}.
Following the second part of the proof of  \cite[Lemma 5.13]{LR24}, we have that $\buchc$ folds to $\hq$, and cell stabilizers map isomorphically to quasiconvex subgroups of $\hg = \pi_1(\hq)$. 
This group is hyperbolic and virtually compact special (see \cite[Lemma 5.12]{LR24}), so cell stabilizers are hyperbolic and virtually compact special too (see \cite[Lemma 5.10]{LR24}).
\end{proof}

\begin{lemma}\label{lem:trivial int stab}
Let $\sigma$ be an $n$-cell of $\buchc$ which is not a branch point, and let $\hat y$ be a branch point of $\buchc$. Then $\stab G\sigma \cap \stab{G}{\hat y} =1$.
\end{lemma}
\begin{proof}
Let $g\in \stab G\sigma \cap \stab{G}{\hat y}$.
We distinguish two cases, but in each case we construct a fixed point for $g$ in $\buchc$ which is not a branch point.
Since the action of $G$ is free away from branch points, this forces $g=1$.

Case 1: $\hat y \not \in \sigma$. Since $\sigma$ is closed and convex in the $\cat{-1}$ space $\buchc$, we can consider the nearest point projection $\pi_\sigma(\hat y)$ of $\hat y$ to $\sigma$. Note that  $\pi_\sigma$ is $g$-equivariant, because $g\in \stab G\sigma$. Hence $g$ fixes $\pi_\sigma (\hat y)$.

Case 2: $\hat y  \in \sigma$. 
Let $p_1,\dots,p_m$ be the $0$-cells of $\sigma$ which are adjacent to $\hat y$ and of minimal distance from $\hat y$ ($m\leq n$).
Then $g$ permutes this collection of points.
Let $z$ be their barycenter inside $\sigma$.
Then $g$ fixes $z$.
\end{proof}


Finally, we show that stabilizers are full relatively quasiconvex subgroups of $(G,\mathcal P)$; see \S\ref{sec:preliminaries} for definitions.

\begin{lemma}\label{lem:quasiconvex_stabilizers}
Let $\sigma$ be a cell of $\buchc$. Then $\stab{G}{\sigma}$ is a full relatively quasiconvex subgroup of $(G,\mathcal P)$. 
\end{lemma}
\proof
By \eqref{item:stab cell branched} in Lemma~\ref{lem:stab cell}, if $\sigma$ is a branch point, then $\stab{G}{\sigma}$ is a maximal parabolic subgroup of $(G,\mathcal P)$ and the statement follows.
So, let us assume that $\sigma$ is not a branch point.
Notice that the action of $G$ on $\buchc$ fits in the setting of Lemma~\ref{lem:relative MilnorSvarc}.
Indeed, it is an action by deck transformations in the complement of the branch points.
Moreover, $\sigma$ is convex in $\buchc$, and $\stab{G}{\sigma}$ acts on it cocompactly (the quotient is a face of $\hq$, which is a compact manifold with boundary).
Therefore, $\stab{G}{\sigma}$ is a relatively quasiconvex subgroup of $(G,\mathcal P)$.
Finally, let $P \in \mathcal P$ be any peripheral subgroup and $g\in G$.
Then  $gPg^{-1}=\stab{G}{\hat y}$ for some branch point $\hat y$, so we conclude by Lemma~\ref{lem:trivial int stab}.
\endproof

\begin{remark}\label{rem: strong 0}
    The above proof actually shows that if a cell stabilizer is not maximal parabolic then it is strongly relatively quasiconvex.
\end{remark}


\section{The dual cubical complex}\label{sec:dual complex}
Following \cite{LR24} we construct a cubical complex that is dual to the stratification induced by the collection of mirrors. This is defined as follows; see Figure~\ref{fig:dccx}.
\begin{itemize}
    \item vertices are $k$-cells in the stratification of $\buchc$,
    
    \item edges represent codimension-1 inclusions,
    
    \item higher-dimensional cubes are glued in whenever their 1-skeleton appears.
\end{itemize}
The resulting cubical complex is denoted by  $\dcc \buchc$.
It comes with a \textit{height} function on the 0-skeleton, recording the dimension of the dual cell.
The action of $G=\pi_1(\relhyp KL)$ on $\buchc$ extends to an action on $\dcc \buchc$ by cubical isometries.
Since $K$ is finite, the action is cocompact. However, it is not proper.
Note that running the same construction on $\uchc$ leads to the $\cat 0$ cubical complex $\dccx$ studied in \cite{LR24}.

\begin{figure}[h]
\centering
\def\svgwidth{.7\columnwidth}
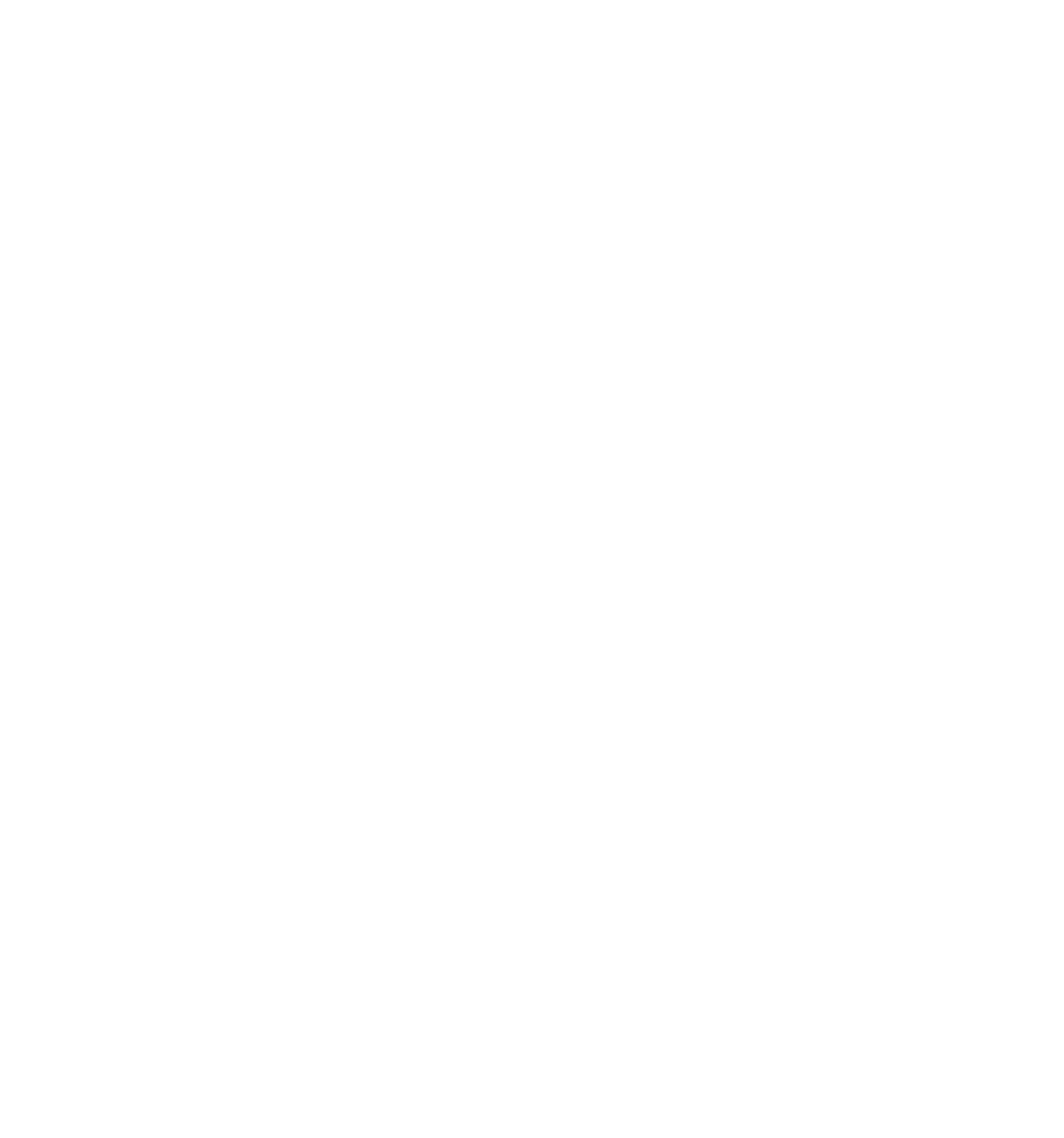
    \caption{The dual cubical complex $\dcc \buchc$ superimposed on the stra\-ti\-fi\-ca\-tion of $\buchc$. Key: $\fullmoon$, $\astrosun$, and $\newmoon$  denote a vertex of height $0$, $1$, and $2$.}    
    \label{fig:dccx}
\end{figure}

\begin{remark}[The case of $L$ with simply connected components, continued]\label{rem:easy case 2}
 An argument analogous to the one in Remark~\ref{rem:easy case} shows that if each component of $L$ is simply connected, then puncturing $\uchc$ at its cone points does not change the fundamental group. 
 Therefore $\buchc = \uchc$, and it follows that  $\dcc \buchc = \dccx$, which is known to be $\cat 0$ by \cite{LR24}.
\end{remark}

Vertices of height at least $2$ in $\dcc \buchc$  do not admit compact neighborhoods.
The vertex dual to a branch point will be called a \textit{branch vertex} of $\dcc \buchc$. 
Branch vertices have height $0$ and are the only vertices of height less than $2$ that do not have compact neighborhoods.
Let $\bv$ denote the collection of branch vertices of $\dcc \buchc$.
Analogously, let $\cv$ denote the collection of \textit{cone vertices} of $\dccx$, i.e. the vertices of $\dccx$ that arise from cone points of $X$.
The branched covering map $\pibuchc: \buchc  \to \uchc$ from \S\ref{sec:auxiliary covering spaces} induces a cubical branched covering map $\dcc \buchc \to \dccx$, see \S\ref{sec:dual npc} for details.

We introduce certain subcomplexes on $\dcc \buchc$, following \cite{LR24}, where the analogous subcomplexes were defined for $\dccx$.
For each tile $\tau$  of $\buchc$, we define the \textit{dual tile} $\dcc \tau$ of $\dcc \buchc$ to be the full subcomplex consisting of vertices dual to cells of $\tau$.
Note that if $v$ is the vertex dual to $\tau$, then $\dcc \tau$ is the cubical $1$-neighborhood of $v$.
Similarly, for each mirror $M$  of $\buchc$, we define the \textit{dual mirror} $\dcc M$ of $\dcc \buchc$ to be the full subcomplex consisting of vertices dual to cells of $M$.


\subsection{\texorpdfstring{$\dcc \buchc$}{The dual cubical complex} is non-positively curved}\label{sec:dual npc}
Thanks to Gromov's link condition, it is enough to check that links of vertices of $\dcc \buchc$ are flag simplicial complexes.
Since the definition of adjacency in the dual cubical complex $\dcc \buchc$ is given in terms of codimension-$1$ inclusion of cells in the stratification of $\buchc$,  the combinatorics of the link of a cell in $\buchc$ completely determine the combinatorics of the link of the dual vertex in $\dcc \buchc$.
Note that the branched covering map $\pibuchc:\buchc \to \uchc$ respects the stratifications of these spaces, in the sense that it sends cells of $\buchc$ to cells of $\uchc$, preserving inclusions. Hence it induces combinatorial maps on the dual cube complexes.

First, let $v \in \dcc \buchc$ be a vertex that is not a branch vertex, and let $\sigma \subseteq \buchc$ be its dual cell.
Note that $\pibuchc:\buchc \to \uchc$ is a covering map in the complement of branch points. 
In particular, it induces an isomorphism  $\lk{\sigma}{\buchc} \to \lk{\pibuchc(\sigma)}{\uchc}$, and therefore an isomorphism $\lk{v}{\dcc \buchc} \to \lk{w}{\dccx}$, where $w$ is the vertex of $\dccx$ dual to the cell $\pibuchc (\sigma)$. The latter link is known to be flag by \cite[Proposition 4.10 (3)]{LR24}.

Let us now consider the link of a branch vertex of $\buchc$.
Recall from  \S\ref{sec:auxiliary covering spaces} that the link of a branch vertex of $\buchc$ is isomorphic to a subdivision of $\widetilde L_i$  for some connected component $L_i$ of $L$.
The argument is similar to that in \cite[Proposition 4.10 (2)]{LR24} (note that a branch vertex has height $0$).
The main difference is that the map

  $$ \cdXbu: \buchc \overset{\pibuchc}{\to} \uchc \overset{\pi}{\to} \hc \overset{\cdX}{\to} X $$
does not induces an isomorphism on links at the branch points.
This is addressed by the following lemma, which is the analogue of \cite[Lemma 3.17]{LR24}.

\begin{lemma}\label{lem:link branch point}
Let $\hat y \in \buchc$ be a branch point, and let $y=\cdXbu(\hat y)\in X$ be the corresponding cone point.
Then the following hold.
\begin{enumerate}
    \item \label{item:link branch point is universal}  
    $\cdXbu$ induces the universal covering map $ \lk{\hat y}{\buchc} \to  \lk {y}{X}$.

    \item \label{item:link branch point is simplicial}   $\lk{\hat y}{\buchc}$ is a piecewise spherical simplicial complex with vertices given by the edges containing $\hat y$, and in which $m+1$ vertices span an $m$--simplex if and only if the corresponding edges are contained in a $(m+1)$--cell.  
\end{enumerate}
\end{lemma}
\begin{proof}
The maps $\piuchc$ and $\cdX$ induce isomorphisms on the link of any vertex. 
Recall from Remark~\ref{rem:cover is coned} that $\buchc$ can be obtained from $\ucrhc KL$ by coning off each copy of the universal cover $\widetilde L_i$ of the components $L_i$ of $L$.
In particular, the map $\pibuchc$ induces the universal covering map on the links of the branch points. 
Therefore we obtain \eqref{item:link branch point is universal}.
To prove \eqref{item:link branch point is simplicial} just argue as in \cite[Lemma 3.17 (3)]{LR24} with $k=0$.
\end{proof}

\begin{lemma}\label{lem:link branch point is flag}
Let $v\in \dcc \buchc$ be a branch vertex. 
Then $\lk{v}{\dcc \buchc}$ is a flag simplicial complex, isomorphic to $\lk{\hat y}{\buchc}$.
\end{lemma}
\begin{proof}
Let $\hat y \in \buchc$ be the branch point dual to $v$, and let $y=\cdXbu(\hat y)\in X$ be the corresponding cone point.
Any vertex $w_i$ of $\lk{v}{\dcc \buchc}$ corresponds to a vertex $v_i$ of $\dcc \buchc$ adjacent to $v$, and therefore to an edge $e_i$ of $\buchc$ that contains $\hat y$.
Note that if $w_i,w_j$ are two adjacent vertices in $\lk{v}{\dcc \buchc}$, then there is a vertex $v_{ij}$ such that $v,v_i,v_j,v_{ij}$ span a square in $\dcc \buchc$. Since $v$ has height $0$, necessarily $v_{ij}$ has height $2$, hence it is dual to a $2$-cell of $\buchc$ containing the edges $e_i,e_j$ dual to $v_i,v_j$.
In particular, $e_i,e_j$ are adjacent in $\lk{\hat y}{\buchc}$.
This shows that  $\lk{v}{\dcc \buchc}$ and  $\lk{\hat y}{\buchc}$ have the same 1-skeleton.

By Lemma~\ref{lem:link branch point}, $\lk{\hat y}{\buchc}$ identifies with the universal cover of $\lk{y}{X}$.
Since $X$ is non-positively curved, $\lk{y}{X}$ is flag, so $\lk{\hat y}{\buchc}$ is flag too.
Therefore, the lemma is proved if we show that $\lk{v}{\dcc \buchc}$ is flag.

Let $w_0,\dots,w_p$ pairwise adjacent vertices in $\lk{v}{\dcc \buchc}$, and let $v_0,\dots,v_p$ be the corresponding vertices of $\dcc \buchc$.
Let $e_0,\dots,e_p$ be the  edges in $\buchc$ that correspond to $v_0,\dots,v_p$. 
These edges contain $\hat y$, and are pairwise adjacent in $\lk{\hat y}{\buchc}$.
Since $\lk{\hat y}{\buchc}$ is flag, there is a cell $\mu$ of $\buchc$ containing $e_0,\dots,e_p$.
The collection of cells that contain $\hat y$ and are contained in $\mu$ give rise to a cube in $\dcc \buchc$ that contains $v_0,\dots,v_p$.
As a result, $w_0,\dots,w_p$ span a simplex in $\lk{v}{\dcc \buchc}$.
\end{proof}

We note that a completely analogous argument shows that the link of a vertex in $\uchc$ and the link of its dual vertex in $\dccx$ are isomorphic.
Also note that it follows from the above discussion that the complex $\dcc \buchc$ is a branched covering of the complex $\dccx$, branching over the set $\bv$ of branched vertices.


\subsection{\texorpdfstring{$\dcc \buchc$}{The dual cubical complex} is simply-connected}
In this section we show that the dual cubical complex is simply connected. 
 Recall that we are working under the standing assumption in Remark~\ref{rem:assumption}.
We follow the approach in \cite{LR24}, which is based on the following two observations about the combinatorial geometry of the dual cubical complex $\dccx$.
\begin{itemize}
    \item[(DT)] An edge-loop entirely contained in a dual tile is nullhomotopic.
    \item[(DM)] If an edge-loop  is not entirely contained in a dual tile, then it can be cut  along dual mirrors and decomposed into a product of edge-loops, each of which is contained in a dual tile.
\end{itemize}

The step (DT) carries over verbatim from \cite[\S 4.2]{LR24}, because the arguments there are completely local, in the sense that they only depend on the geometry and combinatorics of a single tile, and tiles of $\buchc$ are isomorphic to those of $\uchc$.

For the step (DM), we will check that all the arguments from \cite{LR24} carry over to the current setting, because they do not rely in any essential way on the local finiteness of the spaces involved.
Indeed, each point of $\buchc$ is contained in at most finitely many mirrors, and each finite edge-path of $\dcc \buchc$ intersects only finitely many dual mirrors (each of them only finitely many times).

The following statement provides  one of the main properties of mirrors. It is a direct consequence of foldability as in \cite[Proposition 3.14]{LR24}. 

\begin{proposition}\label{prop:mirrors convex}
Each mirror of $\buchc$ is a closed connected convex subspace of $\buchc$.
\end{proposition}

Next, we turn to the separation properties of the collection of mirrors.
Following \cite[\S 3.6]{LR24}, for each $i=1,\dots,n$, let $\bmirrors_i$ be the collection of mirrors of $\buchc$ that fold to one of the two parallel $i^{th}$ faces of $\square^n=[0,1]^n$, i.e. $\{ x_i=0\} $ and $\{ x_i=1\} $. 
Notice that by construction any two elements of $\bmirrors_i$ are disjoint, and even have disjoint $\varepsilon$--neighborhoods for $\varepsilon$ sufficiently small (because $\Gamma$ is cocompact).
Let $\bcomponents_i$ be the collection of connected components of $\buchc \setminus \cup_{M\in \bmirrors_i} M$.
For each mirror $M\in \bmirrors_i$ and for each component $C\in \bcomponents_i$, consider the following \textit{equidistant space}, obtained by pushing the mirror $M$ into the component $C$. 
\[
\bedgespace=\{ x\in C \ | \ d(x,M)=\varepsilon\}.
\]

Note that the local geometry of a component $C\in \bcomponents_i$ in a neighborhood of a mirror $M\in \bmirrors_i$ is not sensitive to the fact that the stratification of $\buchc$ is not locally finite around a branch point.
More precisely, if $\mirrors_i$ and $\components_i$ denote the analogous collections of mirrors and complementary components in $\uchc$, then 
we have that the map $\pibuchc:\buchc\to \uchc$ maps 
each component $C\in \bcomponents_i$ locally  isometrically to a component $\pibuchc(C) \in \components_i$ in $\uchc$. 
So, we obtain the following analogue of \cite[Lemma 3.27]{LR24}.

\begin{lemma}
For $\varepsilon >0$ small enough there is $k\in (-1,0)$ such that the metric induced on $\edgespace$ is $\cat k$. 
\end{lemma}

As a consequence, we obtain the following analogue of \cite[Proposition 3.29]{LR24}.

\begin{proposition}\label{prop:graph of spaces}
$\buchc$ admits the structure of a graph of spaces, with underlying graph a connected tree.
\end{proposition}

The construction and the proof are the same as in the case of $\uchc$.
The only difference is that in the case of $\buchc$ the tree is not locally finite: a mirror that goes through a branch point intersects the closure of possibly infinitely many complementary components, so the corresponding vertex has possibly infinitely many neighbors. 
However, this property is not needed.
The relevant property of this tree is that it has no boundary, and this is still true in our setting, as we now explain.

The arguments  from \cite[\S 3.7]{LR24} are local in nature: they deal with a cell $\sigma$ on a mirror $M$ and a   \textit{framing} for it, i.e. a choice of two tiles  $\tau_1, \tau_2$ such that $\sigma \subseteq \tau_1\cap \tau_2 \subseteq M$.
Once again, the map $\pibuchc:\buchc\to \uchc$ preserves the structure of framings, and this can be used to obtain the following statement, which is analogous to \cite[Proposition 3.37]{LR24}.

\begin{proposition}\label{prop:mirrors separate}
Each $M\in \bmirrors_i$ separates $\buchc$. 
More precisely, let $M\in \bmirrors_i$ be a mirror, let $\sigma \subseteq M$  be a $k$--cell, and let $\{\tau_1,\tau_2\}$ be a framing for $\sigma$. 
Then $\tau_1,\tau_2$ are contained in the closure of two distinct connected components of $\buchc\setminus M$.
\end{proposition}

Arguing as in \cite[\S 4.3]{LR24} it follows that each dual mirror $\dcc M$ separates the dual complex $\dcc \buchc$.
When an edge-loop is not entirely contained in a tile, we want to decompose it into subpaths by cutting it along dual mirrors.
The following definitions are taken from \cite[\S 4.3.2]{LR24}.
Let $p=(v_0,\dots,v_s)$ be an edge--path in $\dcc \buchc$, and let $\sigma_0,\dots, \sigma_s$ be the  cells of $\buchc$ dual to its vertices. 
We say that $p$ is a \textit{bridge} if there exists a mirror $M$ of $\buchc$ such that $v_0,v_s\in \dcc M$, but $p\not \subseteq \dcc M$.
In other words, $\sigma_0,\sigma_s\subseteq M$ but some of the other cells $\sigma_1,\dots,\sigma_{s-1}$ are not contained in $M$. 
In this case, we say that $p$ is \textit{supported} by $M$.
We say $p$ is a \textit{minimal bridge} if none of its subpaths is a bridge

The arguments in \cite[\S 4.3-4]{LR24} do not use the local finiteness of $\uchc$, so they carry over to $\buchc$.
Indeed, they just rely on the orthogonality properties of the collection of mirrors and tiles, such as the fact that either two mirrors are disjoint, or they intersect orthogonally and the projection of one to the other is contained in their intersection; see Lemma~\cite[Lemma 4.22]{LR24}.
This is enough to turn the nearest point projection $\pi_M:\buchc \to M$ to a mirror $M$ into a length-decreasing projection from a certain neighborhood  of $\dcc M$ in $\dcc \buchc$ onto $\dcc M$, see \cite[Lemma 4.24]{LR24} and the discussion after it.
This neighborhood is the one consisting of all the dual tiles that intersect $\dcc M$.
Moreover, if a minimal bridge $p$ is supported by  $M$, then $p$ remains inside this neighborhood of $\dcc M$; see Lemma 4.25 in \cite{LR24}.
It follows that we can project $p$ to $\dcc M$ and obtain a shorter path; see Figure~\ref{fig:projection_minimal_bridge}.
This is stated in the following lemma, which is the core of the step (DM), and is the analogue of \cite[Lemma 4.26]{LR24}.

\begin{lemma}\label{lem:projection control}
Let $p$ be a minimal bridge supported on a mirror $M$. 
Then there exists an edge--path $p^M\subseteq \dcc M$, such that $p^M$ has the same endpoints as $p$ and $\length{p^M} \leq \length p -2$.
\end{lemma}

\begin{figure}[h]
\centering
\def\svgwidth{\columnwidth}
\begingroup%
  \makeatletter%
  \providecommand\color[2][]{%
    \errmessage{(Inkscape) Color is used for the text in Inkscape, but the package 'color.sty' is not loaded}%
    \renewcommand\color[2][]{}%
  }%
  \providecommand\transparent[1]{%
    \errmessage{(Inkscape) Transparency is used (non-zero) for the text in Inkscape, but the package 'transparent.sty' is not loaded}%
    \renewcommand\transparent[1]{}%
  }%
  \providecommand\rotatebox[2]{#2}%
  \newcommand*\fsize{\dimexpr\f@size pt\relax}%
  \newcommand*\lineheight[1]{\fontsize{\fsize}{#1\fsize}\selectfont}%
  \ifx\svgwidth\undefined%
    \setlength{\unitlength}{2544.61889072bp}%
    \ifx\svgscale\undefined%
      \relax%
    \else%
      \setlength{\unitlength}{\unitlength * \real{\svgscale}}%
    \fi%
  \else%
    \setlength{\unitlength}{\svgwidth}%
  \fi%
  \global\let\svgwidth\undefined%
  \global\let\svgscale\undefined%
  \makeatother%
  \begin{picture}(1,0.46359536)%
    \lineheight{1}%
    \setlength\tabcolsep{0pt}%
    \put(0.05159611,0.09543082){\color[rgb]{0,0,1}\makebox(0,0)[lt]{\lineheight{1.25}\smash{\begin{tabular}[t]{l}$\dcc M$\end{tabular}}}}%
    \put(0,0){\includegraphics[width=\unitlength,page=1]{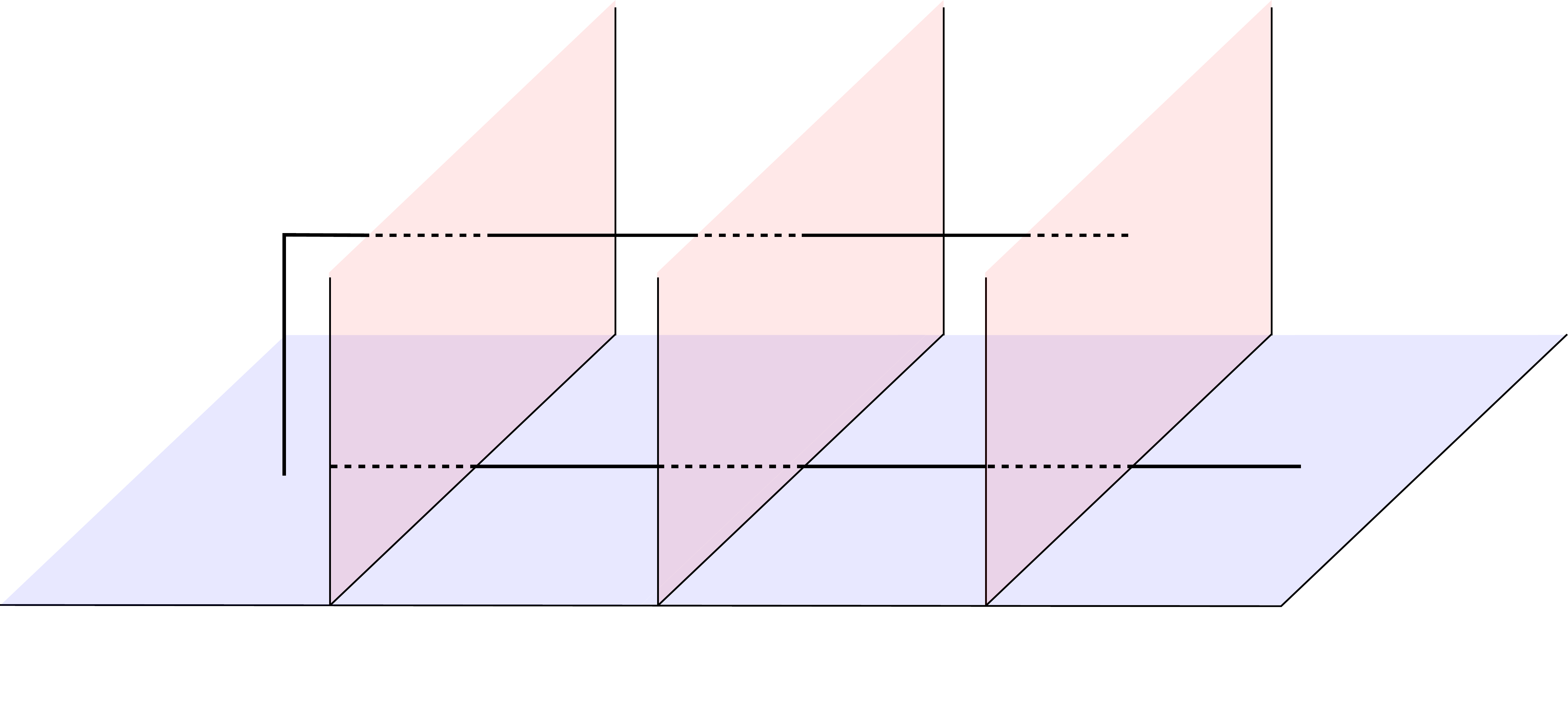}}%
    \put(0.40528107,0.33711708){\color[rgb]{0,0,0}\makebox(0,0)[lt]{\lineheight{1.25}\smash{\begin{tabular}[t]{l}$p$\end{tabular}}}}%
    \put(0.31096457,0.12490456){\color[rgb]{0,0,0}\makebox(0,0)[lt]{\lineheight{1.25}\smash{\begin{tabular}[t]{l}$p^M$\end{tabular}}}}%
    \put(0.88156205,0.34549883){\color[rgb]{0,0,0}\makebox(0,0)[lt]{\lineheight{1.25}\smash{\begin{tabular}[t]{l}$\pi_M$\end{tabular}}}}%
    \put(0,0){\includegraphics[width=\unitlength,page=2]{projection_minimal_bridge.pdf}}%
  \end{picture}%
\endgroup%

    \caption{A minimal bridge $p$ supported by $M$, and its projection  to $\dcc M$.}    
    \label{fig:projection_minimal_bridge}
\end{figure}

We are now ready to prove the main result of this section.

\begin{theorem}\label{thm:dual cubical complex is CAT(0)}
The complex $\dcc \buchc$ is a connected $\cat 0$ cubical complex.
\end{theorem}
\proof
By construction, the complex $\dcc \buchc$ is a path-connected cubical complex. 
Moreover, the link of any vertex is a flag simplicial complex (see Lemma~\ref{lem:link branch point is flag} for the link of a branch vertex, and \cite[Proposition 4.10]{LR24} for the other ones).
So, $\dcc \buchc$ is non--positively curved by Gromov's link condition. 
  
To conclude, we need to show that $\dcc \buchc$ is simply connected.
As in the proof of \cite[Theorem 4.29]{LR24}, we argue that edge--loops are nullhomotopic by induction on their length.
Let $p$ be an edge--loop in $\dcc \buchc$. 
If $p$ does not cross any mirror, then   $p$ stays in a tile and is therefore nullhomotopic.
So let us assume that $p$ crosses a mirror.
Then it must cross it an even number of time. 
Each pair of crossings determines a decomposition of $p$ into two bridges.
Make a choice of a minimal bridge, and use the projection from Lemma~\ref{lem:projection control} to introduce a shortcut along the supporting mirror, which allows us to write $p$ as the product of two shorter edge-loops.
Iterating this process decomposes $p$ into a product of loops that are entirely contained in a dual tile and are therefore nullhomotopic.
\endproof


\subsection{The action of \texorpdfstring{$G$}{G} on \texorpdfstring{$\dcc \buchc$}{the dual cubical complex}}\label{sec:proofs}
We now turn to the study of cube stabilizers.
The idea is to follow the approach in \cite{LR24}, and relate the cube stabilizers for the action of $G=\pi_1(\relhyp KL)$ on $\dcc \buchc$ to the cell stabilizers for the action of $G$ on $\buchc$.
Recall that the action of $G$ on $\ucrhc KL$ by deck transformations extends to an action on $\buchc$ in which each maximal parabolic subgroup identifies with the stabilizer of a branch point.
The following is the leading observation, which follows directly from the definition of the  action of $G$ on $\dcc \buchc$.

\begin{lemma}\label{lem:vertex stab equals cell stab}
The stabilizer of a vertex in $\dcc \buchc$ coincides with the stabilizer of its dual  cell in $\buchc$. 
In particular, the stabilizer of a branch vertex is a  maximal parabolic subgroup.
\end{lemma}


To deal with higher dimensional cubes of $\dcc \buchc$, we observe the following.
By invariance of the height function, the stabilizer of a cube $C$ is always contained in the stabilizer of its vertex of minimal height.
If this minimal vertex is not a branch vertex, then they are actually equal, as established by the following result.
It is obtained as in  \cite[Lemma 5.4]{LR24}, where there are no branch vertices.

\begin{lemma}\label{lem:cube stab equals min vertex stab}
Let $C$ be a cube in $\dccx$.
If the vertex of minimal height $v$ of $C$ is not a branch vertex, then $\stab GC = \stab Gv$.
\end{lemma}

We now consider the case in which $C$  contains a branch vertex.
(Note that if $C$ contains a branch vertex $v$, then $v$ is necessarily the vertex of minimal height.)

\begin{lemma}\label{lem:branched cube stab}
Let $C$ be a cube in $\dccx$ such that the vertex of minimal height $v$ of $C$ is a branch vertex.
\begin{enumerate}

    \item \label{item:stab vertex branched} If $C=v$, then $\stab{G}{C}$ is a maximal parabolic subgroup.

    \item \label{item:stab cube branched} If $C\neq v$, then $\stab GC = 1$.
    
\end{enumerate}
\end{lemma}
\begin{proof}
The proof of \eqref{item:stab vertex branched} just follows from Lemma~\ref{lem:vertex stab equals cell stab} and \eqref{item:stab cell branched} in Lemma~\ref{lem:stab cell}.
So, let us prove \eqref{item:stab cube branched}.
Let $g \in \stab GC$. Since the height function is invariant, $g$ must fix $v$, by uniqueness of the vertex of minimal height. So, $g\in \stab Gv$ too.
If $\sigma$ is the cell dual to $C$ and $\hat y$ is the branch point dual to $v$, then it follows from Lemma~\ref{lem:vertex stab equals cell stab}  that $g\in \stab GC \cap \stab Gv = \stab G\sigma \cap  \stab{G}{\hat y}.$
But this intersection is trivial by Lemma~\ref{lem:trivial int stab}.

\end{proof}


We are now ready to collect the ideas, and prove that the action of $G$ on $\dcc \buchc$ looks like a relatively geometric action in the sense of \cite[Definition 1.1]{EG22} or \cite[Definition 1.9]{GM22}, but with some larger stabilizers ``away from the parabolics''.
Recall from \S\ref{sec:rel hyp} that $\mathcal P$ is a set of representatives of the conjugacy classes of the subgroups $\pi_1(L_i)$, where $L_i$ is a connected component of $L$.

\begin{maintheoremc}{A}\label{thm:action main}
The action of the relatively hyperbolic group $G=\pi_1(\relhyp KL)$ on the $\cat 0$ cubical complex $\dcc \buchc$ satisfies the following properties.
\begin{enumerate}
    \item  $\leftQ{\dcc \buchc}{G}$ is compact.
    \item Each $P \in \mc{P}$ acts elliptically on $\dcc \buchc$.
    
    \item For each cube $C$ of $\dcc \buchc$, $\stab{G}{C}$ is either maximal parabolic, or else is full relatively quasi-convex, hyperbolic, and virtually compact special. 
\end{enumerate}
\end{maintheoremc}
\begin{proof}
First of all, $\dcc \buchc$ is $\cat 0$ by Theorem~\ref{thm:dual cubical complex is CAT(0)}.
The action of $G$ on it is cocompact because $K$ is finite.
Each $P\in \mathcal P$ fixes a branch point in $\buchc$, and therefore it fixes the dual vertex in $\dcc \buchc$.
This proves the first two statements, so let us prove the third one.
Let $C\subseteq \dcc \buchc$ be a cube.
There are three cases to consider (recall that if a cube contains a branch vertex, then that vertex is the vertex of minimal height).

First, if $C$ is a branch vertex, then $\stab GC$ is a  maximal parabolic subgroup of $G$ by \eqref{item:stab vertex branched} in Lemma~\ref{lem:branched cube stab}. 
Next, consider the case that $C$ is not a branch vertex but its vertex of minimal height $v$ is a branch vertex.
    By \eqref{item:stab cube branched} in Lemma~\ref{lem:branched cube stab} we have that $\stab GC =1$. 
 Finally, suppose that $C$ is not a branch vertex and its vertex of minimal height $v$ is not a branch vertex.
    In this case, $\stab GC = \stab Gv$ by Lemma~\ref{lem:cube stab equals min vertex stab}.
    Moreover, by Lemma~\ref{lem:vertex stab equals cell stab} this also equals $\stab G\sigma$, where $\sigma $ is the cell of $\buchc$ dual to $v$.
    By \eqref{item:stab cell unbranched} in Lemma~\ref{lem:stab cell} we know that this is a hyperbolic and virtually compact special group.
    Moreover, by Lemma~\ref{lem:quasiconvex_stabilizers} we also know that it is a full relatively quasiconvex subgroup of $(G,\mathcal P)$. 
    This concludes the proof.
\end{proof}

\begin{remark}\label{rem: strong}
It follows from Remark~\ref{rem: strong 0} that if a cube stabilizer is not a maximal parabolic subgroup then it is strongly relatively quasiconvex.
\end{remark}

\begin{maintheoremc}{B}\label{thm: res fin special main}
Let $G=\pi_1(\relhyp KL)$ and $\mathcal P$ be as above.
Then the following hold.
\begin{enumerate}
    \item \label{item: res fin main} If each $P\in \mathcal P$ is residually finite, then $G$ is residually finite and each $P\in \mathcal P$ is separable.

    \item  \label{item: special main} If each $P\in \mathcal P$ is hyperbolic and virtually compact special, then $G$ is hyperbolic and virtually compact special.
\end{enumerate}
\end{maintheoremc} 
\begin{proof}
By Theorem~\hyperref[thm:action main]{A}, the pair $(G,\mathcal P)$ satisfies the conditions in Theorem~\ref{thm: appendix main} from the Appendix. 
Since the trivial subgroup and the peripheral subgroups are full relatively quasiconvex in $(G,\mathcal P)$, the statement in \eqref{item: res fin main} follows.

To prove \eqref{item: special main}, 
we argue as follows.
First of all, if a group is hyperbolic relative to a hyperbolic subgroup then it is itself hyperbolic, see \cite[Corollary 2.41]{OS06}.
So, we have that $G$ is hyperbolic.
Consider the action of $G$ on 
 the $\cat 0$ cubical complex $\dcc{\buchc}$ and let $H$ be a non-trivial cube stabilizer.
We claim that $H$ is quasiconvex in $G$ and virtually compact special.
Indeed, by Theorem~\hyperref[thm:action main]{A} we have two cases.
If $H$ is a maximal parabolic, then it is quasiconvex by \cite[Corollary 8.2]{HR10}, and virtually compact special by assumption.
Otherwise, $H$ is virtually compact special, and strongly relatively quasiconvex in $(G,\mathcal P)$; see Remark~\ref{rem: strong}. 
By \cite[Theorem 1.9]{OS06} $H$ is quasi-isometrically embedded in $G$, hence $H$ is quasiconvex, since $G$ is hyperbolic.
It follows that the action of $G$ on $\dcc{\buchc}$ satisfies the conditions in \cite[Theorem D]{GM23}, so $G$ is virtually compact special.
\end{proof}


\section{Applications to manifolds}\label{sec:applications}
We collect some applications to the study of aspherical manifolds in \S\ref{sec:new examples}, cobordism of manifolds in \S\ref{sec: cobordism}, and (non-)triangulability of manifolds in \S\ref{sec: triangulability}.

\subsection{Aspherical manifolds with residually finite fundamental groups}\label{sec:new examples}
The purpose of this section is to obtain 
examples of closed aspherical manifolds whose fundamental groups have interesting algebraic properties.
In Theorem~\ref{thm: new main thm} we construct closed aspherical manifolds in each dimension $n\geq 6$, whose fundamental group is residually finite.
These examples can be chosen to be Riemannian and non-positively curved, and even negatively curved for $n\geq 9$; see Remark~\ref{rmk: new neg curved}.
These examples are new, in the sense that they are not homotopy equivalent to manifolds for which residual finiteness of the fundamental group was previously known by other methods.
Finally, in Theorem~\ref{thm: new special} we construct negatively curved Riemannian manifolds of dimension $n\geq 5$, which are not locally symmetric and whose fundamental groups are virtually compact special.

\bigskip
All of the examples constructed in this section are based on a procedure that associates to a closed triangulable manifold $M$ another closed triangulable manifold $\hypmt M$, which contains $M$ as a codimension-1 submanifold. 
We call  $\hypmt M$  the \textit{hyperbolized mapping torus} of $M$.
The construction of $\hypmt M$ is as follows, see Figure~\ref{fig:hypmt}.
Let $K_0$ be a triangulation of $M$.
Extend $K_0$ to a triangulation $K_1$ of $M\times [0,1]$. 
Glue two copies of $K_1$ via the identity on $M\times \{0\}$ to obtain a triangulation $K$ of $M\times [-1,1]$.
Denote by $L$ the boundary of $K$; it consists of two components, each homeomorphic to $M$.
Notice that $j:K\to K, (m,t)\mapsto (m,-t)$ is a simplicial involution exchanging the two components of $L$.
Now, apply relative strict hyperbolization to the pair $(K,L)$ to obtain a compact manifold  with boundary $\relhyp {K}{L}$ whose boundary $\partial \relhyp {K}{L}$ is homeomorphic to $M\times \{\pm 1\}$.
(Alternatively, one could apply strict hyperbolization to the simplicial suspension of $K_0$, and then remove sufficiently small open neighborhoods of the cone points.)
Finally, let $\hypmt M$ be the closed manifold obtained by gluing the two boundary components of $\relhyp {K}{L}$ together.
By construction,  $\hypmt M$ contains  a $\pi_1$-injective codimension-1 submanifold homeomorphic to $M$ arising from the two boundary components of $\relhyp {K}{L}$ that have been glued together. 
We keep referring to this submanifold as $M$.

\begin{figure}[h]
\centering
\def\svgwidth{.8\columnwidth}
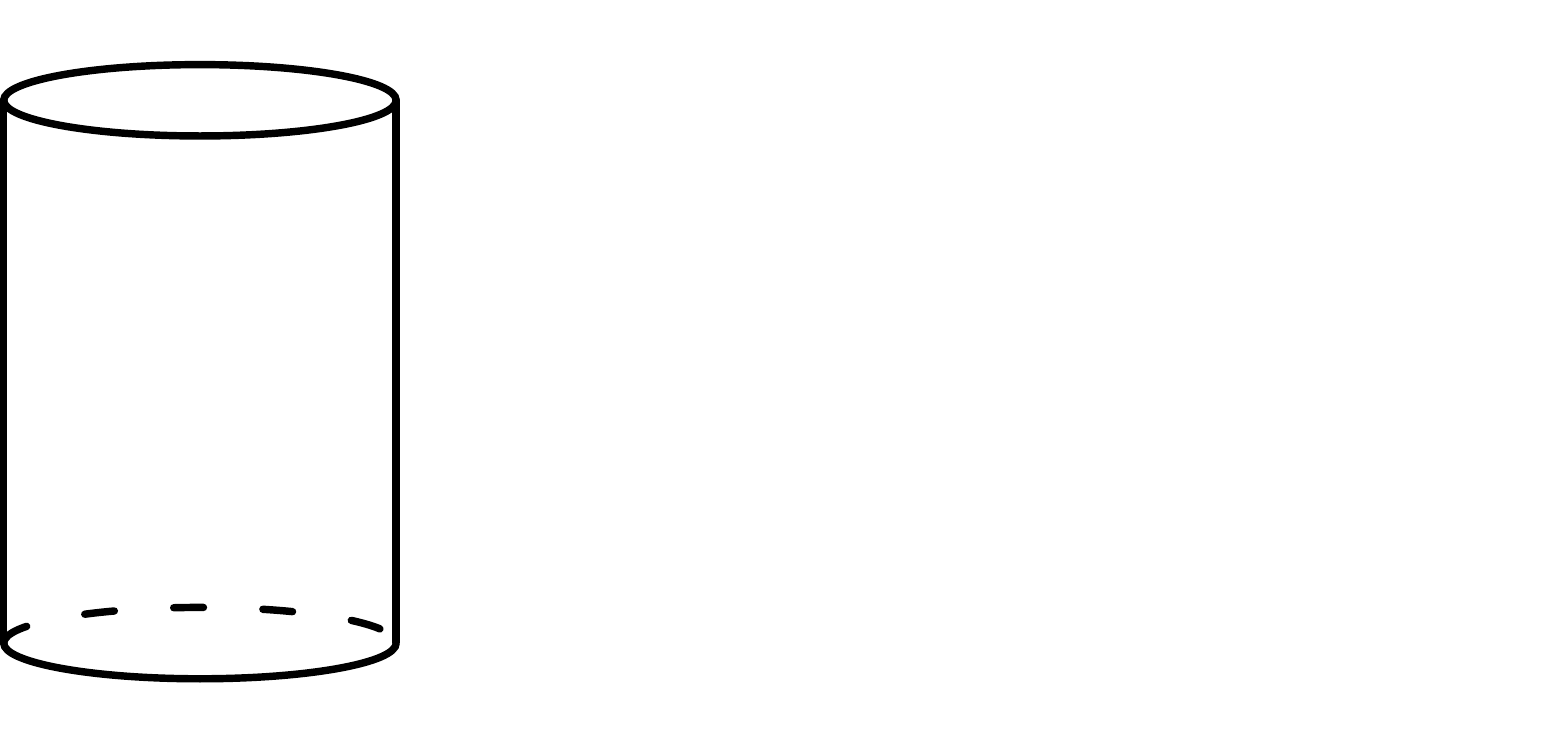
    \caption{The construction of the hyperbolized mapping torus $\hypmt M$.}
    \label{fig:hypmt}
\end{figure}

\begin{lemma}\label{lem: new res fin}
    If $\pi_1(M)$ is residually finite, 
    then $\pi_1(\hypmt{M})$ is residually finite.
\end{lemma}
\begin{proof}
First of all, note that $\pi_1(\hypmt{M})$ splits as a nice HNN extension.
Indeed, the involution $j:K\to K$ defined above induces an automorphism  $j_\ast$ of $G=\pi_1(\relhyp KL)$ that exchanges the subgroups $H_1,H_2$ corresponding to the two boundary components of $\relhyp KL$.
So, $\pi_1(\hypmt M)$ can be presented as the HNN extension  
$$\langle t, G \mid tht^{-1} = j_\ast (h), h \in H_1 \rangle.$$
The problem of residual finiteness for HNN extensions of this type was considered by Baumslag--Tretkoff in \cite[Lemma 4.4]{BT78}.
More precisely, they proved that an HNN extension induced by an automorphism of $G$ as above is residually finite as soon as $G$ and $H_1$ satisfy the following two properties: (1) $G$ is finitely generated and residually finite; (2) for any $x_1,\dots,x_n \in G\setminus H_1$ there is a normal subgroup $N$ of finite index in $G$ such that $x_iH_1\cap N = \varnothing$ for all $i=1,\dots, n$.
We will now verify that these two properties are satisfied in our case.

To check (1), we argue as follows. 
We know that $G$ is finitely generated and  hyperbolic relative to  $\mathcal P = \{H_1,H_2\}$.
Notice that $H_1,H_2$ are both isomorphic to $\pi_1(M)$, which is assumed to be residually finite. 
So, by Theorem~\hyperref[thm: res fin special main]{B} we have that $G$ is residually finite.

Next, note that (2) is satisfied as soon as $H_1$ is separable in $G$.
To see this, let $x_1,\dots,x_n \in G\setminus H_1$.
If $H_1$ is separable in $G$, then for each $i=1,\dots, n$ we can find a normal subgroup of finite index $N_i\leq G$ such that $H_1\leq N_i$ and $x_i\notin N_i$.
Let $N=\cap N_i$.
Then $N$ is still normal and of finite index in $G$. 
Moreover, $H_1\leq N$, and thus $x_iH_1\subseteq x_iN$.
Since $x_i\notin N$, it follows that $x_iN\cap N = \varnothing$, hence $x_iH_1\cap N=\varnothing$ for all $i=1,\dots, n$, as desired.
To conclude, note that our $H_1$ is indeed separable in $G$ thanks to Theorem~\hyperref[thm: res fin special main]{B} and the assumption that $\pi_1(M)$ is residually finite.
\end{proof}

\begin{remark}[Smoothness]\label{rem: new smooth}
When $M$ is smooth, one can ensure that $\hypmt{M}$ is smooth too, by applying strict hyperbolization with a sufficiently large hyperbolizing cell to a smooth triangulation $K$ of $M\times [-1,1]$ as in \cite{O20}.
Moreover, if $M$ admits a non-positively (resp. negatively) curved Riemannian metric, then $\hypmt{M}$ admits a Riemannian metric of non-positive (resp. negative) curvature in which $M$ embeds totally geodesically; see \cite{NP13} for details.
The sectional curvatures of $\hypmt M$ can be pinched arbitrarily close to $-1$ along planes that are sufficiently far away from the tangent bundle to $M$.
While smoothness is not essential in the following discussion, we will work in this Riemannian setting for convenience.
\end{remark}

We now use the above construction to obtain the desired examples. 
Let $\Lambda$ be a torsion-free uniform lattice in $\slr 3$ and let  $M=\doubleQ{\slr 3}{\sor 3}{\Lambda}$ be the corresponding locally symmetric space.
Then $M$ is a closed non-positively curved Riemannian manifold of dimension $5$, whose fundamental group $\pi_1(M)=\Lambda$ is residually finite and has property (T).
Let $N_5=M$ and recursively define $N_{n}=\hypmt{N_{n-1}}$.

\begin{theorem}\label{thm: new main thm}
  For all $n\geq 6$, $N_n$ is a closed Riemannian $n$-manifold 
 of non-positive curvature with residually finite fundamental group and not homotopy equivalent to any of the following:
    \begin{enumerate}
        \item \label{item:loc sym} a locally symmetric space,
        \item \label{item:gt} a Gromov-Thurston manifold,
        \item \label{item: hypman} a strictly hyperbolized manifold,
        \item \label{item: kahler} a closed K\"ahler manifold.
    \end{enumerate}
\end{theorem}

Here, a \textit{Gromov--Thurston manifold}  is one of the examples constructed in \cite{GT87}, and a \textit{strictly hyperbolized manifold} is the result of applying Charney-Davis strict hyperbolization or Ontaneda's Riemannian hyperbolization to any closed triangulable manifold.
Moreover, recall that all the examples of negatively curved manifolds constructed by Mostow--Siu in \cite{MS80}, Deraux in \cite{DE05}, or Stover--Toledo in \cite{ST22,ST22n} are K\"ahler manifolds.

\begin{proof}
First of all, by construction $N_n$ is a closed non-positively curved Riemannian manifold of dimension $n$ that contains $N_{n-1}$ as a totally geodesic codimension-1 submanifold; see Remark~\ref{rem: new smooth}.
In particular, since $N_5=M$, we have that $\pi_1(N_n)$ contains a subgroup isomorphic to $\Lambda=\pi_1(M)$. 
Moreover, $\pi_1(N_n)$ is residually finite by Lemma~\ref{lem: new res fin}.

By construction $\pi_1(N_n)$ splits as an HNN extension, hence it acts on the associated Bass-Serre tree without a global fixed point.
In particular, it does not have property (FA), hence it does not have property (T), see \cite{WA82}. 
It follows that $\pi_1(N_n)$ cannot be a lattice in a Lie group of higher rank, nor in the isometry group of a quaternionic hyperbolic space or the Cayley hyperbolic plane, because all of these Lie groups have property (T). 
To prove \eqref{item:loc sym} we need to deal with the possibility that $\pi_1(N_n)$ is a lattice  in the isometry group of a real or complex hyperbolic space. 
If this were the case, we would obtain a free isometric action of $\Lambda \subseteq \pi_1(N_n)$ on a real or complex hyperbolic space.
But this would be impossible: since  $\Lambda$ has property (T), any isometric action on a real or complex hyperbolic space must have a fix point by Corollary 23 in \cite{DV89}. 
This proves \eqref{item:loc sym}.

To prove \eqref{item:gt} and \eqref{item: hypman} we argue as follows.
Suppose by contradiction that $N_n$ is homotopy equivalent to a Gromov-Thurston manifold or a strictly hyperbolized manifold. 
Then $\pi_1(N_n)$ acts geometrically on a finite dimensional $\cat 0$ cubical complex by \cite{GI17} or \cite{LR24} respectively.
In particular, we obtain a proper cubical action of $\Lambda$ on a $\cat 0$ cube complex.
Since $\Lambda$ has property (T), this is in contradiction with Theorem B in \cite{NR97}.

Finally, we prove \eqref{item: kahler}. 
Suppose by contradiction that $N_n$ is homotopy equivalent to a closed K\"ahler manifold $X$.
Since $N_n$ is non-positively curved, $N_n$ and $X$ are 
 actually homeomorphic by \cite{BL12}.
Then $N_n\setminus N_{n-1}$ is homeomorphic to an open subset of $X$.
But by definition of $N_n$ as $\hypmt{N_{n-1}}$, we have that $N_n\setminus N_{n-1}$ is the interior of a manifold  obtained by relative strict hyperbolization on $N_{n-1}\times [-1,1]$, so we reached a contradiction with \cite[Theorem 10.5]{BE07}.
\end{proof}

\begin{remark}
The crucial property of $M$ that we have used above is that $M$ is a closed aspherical manifold whose fundamental group is residually finite and contains a subgroup $\Lambda$ with property (T).
Notably, these properties are preserved when passing from $M$ to its hyperbolized mapping torus $\hypmt M$.
So, given any such $M$ one can recursively define a sequence of manifolds as in Theorem~\ref{thm: new main thm} and certain other prescribed features, as in Remark~\ref{rmk: new neg curved} below.
\end{remark}

\begin{remark}[Negatively curved examples]\label{rmk: new neg curved}
Let $\Lambda '$ be a torsion-free uniform lattice in the isometry group of the quaternionic hyperbolic plane $\spisom 2 = \spr 2$ and let 
$M '= \quathyp 2/\Lambda '$ be the associated $8$-dimensional manifold.
Now, let $N_8'=M'$ and recursively define $N_{n}'=\hypmt{N_{n-1}'}$.
Then for all $n\geq 9$, the manifold $N_n'$ satisfies the conclusion of Theorem~\ref{thm: new main thm}, with the added feature of negative curvature; see Remark~\ref{rem: new smooth}.
In particular, $\pi_1(N'_n)$ is a residually finite hyperbolic group.
\end{remark}
\begin{remark}[Property (T) vs Haagerup property]
    Note that the above constructions provide examples of one-ended $\cat 0$ and hyperbolic groups that do not have property (T), as they split as HNN extensions, and do not have the Haagerup property, as they contain an infinite subgroup with (T).
    On the other hand, all these groups have relative property (T) with respect to that subgroup;  see \cite[\S 1.4]{BHV08}.
\end{remark}

The negatively curved examples from Remark~\ref{rmk: new neg curved} do not have cubulated fundamental group, because they are constructed starting with a manifold $M$ whose fundamental group has property (T).
On the other hand, if one starts with a manifold $M$ whose fundamental group is cubulated (e.g., a real hyperbolic manifold defined by a uniform real hyperbolic lattice of simple type; see \cite{HW12}), then one can obtain negatively curved examples whose fundamental groups are cubulated too.
This is based on the following lemma, which is the analogue of Lemma~\ref{lem: new res fin}.

\begin{lemma}\label{lem: new special}
   If $M$ be a negatively curved Riemannian manifold and $\pi_1(M)$ is virtually compact special, then $\hypmt M$ is a negatively curved Riemannian manifold and $\pi_1(\hypmt M)$ is virtually compact special.
\end{lemma}
\begin{proof}
As in the proof of Lemma~\ref{lem: new res fin}, 
$G=\pi_1(\relhyp KL)$ is hyperbolic relative to $\mathcal P=\{H_1,H_2\}$ and $\pi_1(\hypmt M)$ can be represented as the HNN extension of $G$ with respect to $H_1$.
Note that $H_i\cong \pi_1(M)$ is hyperbolic and virtually compact special, so by Theorem~\hyperref[thm: res fin special main]{B} we have that $G$ is hyperbolic and virtually compact special.
By Remark~\ref{rem: new smooth} we know that $\hypmt M$ admits a negatively curved Riemannian metric in which $M$ embeds as a totally geodesic submanifold.
In particular, $\pi_1(\hypmt M)$ is hyperbolic and $H_1$ is quasiconvex.
Then by \cite[Theorem 13.1]{W21} we have that $\pi_1(\hypmt M)$ is virtually compact special.
    \end{proof}

This can be used to construct examples of Riemannian manifolds of  negative curvature with cubulated fundamental group which are not real hyperbolic.
Note that manifolds with the same properties can also be obtained directly via strict hyperbolization (see \cite{O20,LR24}).
We suspect that such manifolds are different from the examples in Theorem~\ref{thm: new special}, but we are not aware of invariants that can  distinguish the two classes.

\begin{theorem}\label{thm: new special}
 The hyperbolized mapping torus construction $M\mapsto \hypmt M$ can provide for all $n\geq 5$   a closed Riemannian $n$-manifold $N''_n$  such that:
      \begin{enumerate}
          \item \label{item: new ex special neg} $N''_n$ has negative sectional curvature.
          \item \label{item: new ex special vcs} $\pi_1(N_n'')$ is virtually compact special.
          \item \label{item: new ex special loc sym} $N_n''$ is not homotopy equivalent to a locally symmetric space or K\"ahler manifold.
      \end{enumerate}
\end{theorem}
\begin{proof}
Let $M$ be a negatively curved Riemannian $4$-manifold with virtually compact special fundamental group that is not a real hyperbolic manifold, such as a Gromov--Thurston manifold or a strictly hyperbolized manifold. 
These are known to have virtually special fundamental group by \cite{GI17} and \cite{LR24} respectively.
Let $N''_4=M$ and for $n\geq 5$ let $N''_n=\hypmt{N_{n-1}}$.
By Lemma~\ref{lem: new special} we have that $N''_n$ is a negatively curved Riemannian manifold and that $\pi_1(N''_n)$ is virtually compact special, which proves \eqref{item: new ex special neg} and \eqref{item: new ex special vcs}.

To prove \eqref{item: new ex special loc sym} we argue as follows. 
As in Theorem~\ref{thm: new main thm}, we have that  $\pi_1(N_n'')$ does not have property (T) and $N_n''$ is not K\"ahler. 
So, we only need to check that $N_n''$  cannot be a real hyperbolic manifold.
By contradiction, suppose that 
$\pi_1(N_n'')$ is isomorphic to a real hyperbolic lattice $\Gamma \subseteq \ishyp n$.
Since $N_{n-1}''$ is a totally geodesic submanifold of $N_n''$, its fundamental group identifies with a quasiconvex subgroup $H \subseteq \Gamma$.

By construction, $N_n''$ admits an involution that 
fixes $N_{n-1}''$ pointwise.
This provides an outer automorphism $\phi$ of $\Gamma$ that preserves $H$.
By Mostow rigidity, $\phi$ is realized by an  involution $g \in \ishyp n$.
Since $g$ fixes a totally geodesic $\hh^k \subseteq \hh^n$ for some $k\leq n-1$, 
the fixed set $\operatorname{Fix}_\infty(g)$ for the induced action of $g$ on $\partial_\infty \hh^n$ is a round sphere $S^{k-1}=\partial_\infty \hh^k$.
The Gromov boundary of $H=\pi_1(N_{n-1}'')$ is  $S^{n-2}$ because $N_{n-1}''$ is a negatively curved Riemannian $(n-1)$-manifold.

The limit set $\Lambda(H)=S^{n-2}$ of $H$ in $\partial_\infty \hh^n$ is a priori just a topologically embedded sphere (not necessarily a round one).
However, since $\phi$ preserves $H$, we have $\Lambda(H)=S^{n-2} \subseteq \operatorname{Fix}_\infty(g)=S^{k-1}$. 
Therefore we must have $\Lambda(H)=\operatorname{Fix}_\infty(g)=S^{n-2}$.
This implies that $\pi_1(N_{n-1}'')=H$ acts geometrically on the totally geodesic $\hh^{n-1}$ stabilized by $g$, and therefore  that $N_{n-1}''$ admits a real hyperbolic structure too.
One can repeat this argument all the way down to $M$ and obtain a real hyperbolic structure on $M$, which provides a contradiction.
\end{proof}

\subsection{Cobordism of manifolds}\label{sec: cobordism}
Some of the most classical applications of hyperbolization procedures are to cobordism of manifolds.
For instance,  any closed triangulable manifold $M$  is cobordant to $\hyperbolize{M}$, which is a closed aspherical manifold with negative curvature and  virtually compact special fundamental group; see \cite{CD95,LR24}.
Moreover, if $M$ is smooth, then $\hyperbolize M$ can be chosen to be smooth and have pinched negative sectional curvatures; see \cite{O20}.
The following statements provide more properties about the fundamental group of the cobordism.
It applies in particular when $M_1$ is any closed  triangulable manifold and $M_2=\hyperbolize{M_1}$ 
 is its strict hyperbolization.

\begin{corollary}\label{cor: cobordism}
    Let $M_1,M_2$ be two closed triangulable manifolds with residually finite fundamental group.
    Suppose there is a triangulable cobordism between them.
Then there exists a compact Riemannian manifold with boundary $W$ such that
    \begin{enumerate}
    \item $W$ is a cobordism between $M_1$ and $M_2$
    \item $\pi_1(W)$ is relatively hyperbolic relative to  $\{\pi_1(M_1), \pi_1(M_2)\}$
    \item  $\pi_1(W)$ is residually finite and  $\pi_1(M_i)$ is separable.
\end{enumerate}
\end{corollary}
\begin{proof}
Let $W_0$ be a triangulable cobordism between $M_1$ and $M_2$, 
and let $W=\relhyp{W_0}{M_1 \cup M_2}$.
By construction $\partial W \cong M_1 \cup M_2$ and its fundamental group is relatively hyperbolic with respect to the fundamental groups of the boundary components; see \cite{BE07}.
The statement follows from \eqref{item: res fin main} in Theorem~\hyperref[thm: res fin special main]{B}. 
\end{proof}

\begin{remark}[Variations]
If $M_1$ and $M_2$ are smooth and smoothly cobordant, then $W$ can be taken to be smooth by \cite{O20}.
If $\pi_1(M_i)$ is hyperbolic and virtually compact special, then by \eqref{item: special main} in Theorem~\hyperref[thm: res fin special main]{B} we have that $\pi_1(W)$ is also hyperbolic and virtually compact special.
\end{remark}

 
We also propose the following more geometric application.
Let $N$ be a non-compact complete Riemannian manifold $N$ of finite volume.
A closed Riemannian manifold $M$ \textit{bounds} $N$ \textit{geometrically}
if there is a bounded set $B\subset N$ such that $N\setminus B$ is diffeomorphic to $M\times [0,+\infty)$.
As shown by Long and Reid in \cite{LR00} there are closed flat manifolds that do not bound geometrically a real hyperbolic manifold.
On the other hand, Ontaneda proved that every closed flat manifold bounds geometrically a manifold of pinched negative curvature in \cite{O20}.
The following statement provides information about its fundamental group.

\begin{corollary}\label{cor: flat}
Let $\varepsilon>0$ and let $M$ be a closed  flat manifold.
Then there exists a complete Riemannian manifold of finite volume $N$ such that
\begin{enumerate}
    \item  $N$ has sectional curvatures in $[-1-\varepsilon,-1]$.
    \item $M$ bounds $N$ geometrically.
    \item $\pi_1(N)$ is relatively hyperbolic relative to  $\pi_1(M)$. 
    \item $\pi_1(N)$ is residually finite and  $\pi_1(M)$ is separable.
\end{enumerate}
\end{corollary}
\begin{proof}
The first two statement are due to Ontaneda, see \cite[Corollary 7]{O20}. 
They are based on the fact that a closed flat manifold bounds smoothly a smooth compact manifold (see \cite{HR82}), so that hyperbolization can be applied.
The third one is due to Belegradek; see \cite{BE07}.
For the last statement, note that a flat manifold is virtually a torus by Bieberbach's Theorem.
In particular, $\pi_1(M)$ is virtually abelian hence residually finite.
Then we conclude again by \eqref{item: res fin main} in Theorem~\hyperref[thm: res fin special main]{B}.
\end{proof}

As in \cite{O20}, one can  get a similar statement for an almost flat manifold. In this case one needs to additionally assume that the manifold bounds smoothly and has residually finite fundamental group.


\subsection{Aspherical manifolds that cannot be triangulated}\label{sec: triangulability}
Manolescu has shown in \cite{MA16} that for each $n\geq 5$ there is a  closed topological $n$-manifold that cannot be triangulated, i.e., is not homeomorphic to a simplicial complex.
Davis, Fowler, and Lafont  showed in \cite{DFL14} that for $n\geq 6$ one can take such a manifold to be aspherical and have hyperbolic fundamental group.
We now show that the fundamental groups of these manifolds are cubulated.

\begin{theorem}\label{thm: new no tri}
    For each $n\geq 6$ there is a closed aspherical $n$-manifold $N^n$ that cannot be triangulated and whose fundamental group is hyperbolic and virtually compact special.
\end{theorem}
\begin{proof}
    The manifolds are the ones constructed in [\S 3, pp. 800-801]\cite{DFL14}, by applying suitable strict hyperbolization procedures to a construction by Galewski and Stern in \cite{GS79}.
    We consider the case $n=6$. 
    Higher dimensional examples are obtained in a similar way by taking product with tori of the initial building pieces.
    The manifold $N^6 $ is a closed aspherical $6$-manifold obtained by gluing two suitable manifolds with boundary $P_1$ and $P_2$ along their boundary, so $\pi_1(N^6)$ splits as the associated amalgamated free product.

The manifold with boundary $P_1$ is obtained via strict hyperbolization (i.e., $P_1=\hyperbolize{P'}$ for a certain triangulable manifold with boundary $P'$) and $P_2$ is obtained via relative strict hyperbolization (i.e., $P_2=\relhyp{W}{\partial P_1}$ for a certain triangulable manifold $W$ with boundary  $\partial W=\partial P_1$).
A priori, the former has hyperbolic fundamental group and the latter has relatively hyperbolic fundamental group.
However,  since $\partial P_1=\partial \hyperbolize{P'}=\hyperbolize{\partial P'}$ is a  closed manifold, we have that $\pi_1(\partial P_1)$ is a hyperbolic group.
As observed in \cite{DFL14}, it follows that
$\pi_1(N^6)$ is hyperbolic.
Moreover, we also have that $\pi_1(\partial P_1)$ is quasiconvex in $\pi_1(N^6)$, since it is quasiconvex in both factors; see \cite{KM98,PA93}.
Thanks to \cite[Theorem 13.1]{W21}, we are left to show that $\pi_1(P_1)$ and $\pi_1(P_2)$ are virtually compact special.

Since $\partial P_1=\hyperbolize{\partial P'}$ is the strict hyperbolization of a closed manifold,  it follows from \cite{LR24} that $\pi_1(\partial P_1)$ is virtually compact special.
Then $\pi_1(P_2)$ is virtually compact special by \eqref{item: special main} in Theorem~\hyperref[thm: res fin special main]{B}.
Finally, let us consider the manifold $P''$ obtained by doubling $P'$ along its boundary.
The fundamental group of its strict hyperbolization $\hyperbolize{P''}$ is hyperbolic, and also virtually compact special  by \cite{LR24}. 
Since $P'$ is a subcomplex of $P''$ (with respect to any triangulation), the inclusion $P_1=\hyperbolize{P'} \hookrightarrow \hyperbolize{P''}$ is a local isometry. 
It follows that $\pi_1(P_1)$ is a quasiconvex subgroup of $\pi_1(\hyperbolize{P''})$, and therefore it is virtually compact special too.
\end{proof}

\begin{remark}[Virtual triangulability]
For a topological manifold $M$ of dimension at least $5$ the Kirby-Siebenmann class  $\Delta(M) \in H^4(M,\zz_2)$ is an obstruction to the existence of a PL-structure on $M$, in the sense that $M$ admits a PL-structure if and only if $\Delta(M)=0$, see \cite[Theorem 5]{GS80}.
Similarly, in dimension at least $6$ there is a class $\delta(\Delta(M))\in H^5(M,\ker (\mu))$ such that $M$ admits a triangulation if and only if $\delta(\Delta(M))=0$.
Here $\mu: \Theta^3 \to \zz_2$ is the Rokhlin homomorphism for the homology cobordism group $\Theta^3$, and $\delta:H^4(M,\zz_2) \to H^5(M,\ker (\mu))$ is the connecting homomorphism associated to the short exact sequence
$$0 \to \ker (\mu) \to \Theta^3 \overset{\mu}{\to} \zz_2\to 0.$$
It turns out that if $Sq^1(\Delta(M))\neq 0$ then $\delta(\Delta(M))\neq 0$, where $Sq^1:H^4(M,\zz_2) \to H^4(M,\zz_2)$ is the first Steenrod square.

The manifold $N^n$ from Theorem~\ref{thm: new no tri} has residually finite fundamental group, hence it admits a lot of finite covers.
While $N^n$ is not triangulable, one can ask if it is \textit{virtually triangulable}, i.e., if it admits a finite cover that is triangulable.
We note that a triangulable cover must have even degree.
Indeed, if $\pi:\widehat N^n_d \to N^n$ is a  cover of odd degree $d$, then the induced map on cohomology $\pi^*:H^k(N^n,\zz_2)\to H^k(\widehat N^n_d,\zz_2)$ is injective.
It is shown in \cite{DFL14} that  $Sq^1(\Delta(N^n))\neq 0$.
By naturality of the Kirby-Siebenmann class and the Steenrod square, it follows that $Sq^1(\Delta(\widehat N^n_d))\neq 0$, so $\widehat N^n_d$ does not admit a triangulation. 

On the other hand, for a cover $\pi:\widehat N^n_d \to N^n$ of even degree $d$, the map $\pi^*$ can have non-trivial kernel, so it is not clear whether either $\Delta(\widehat N^n_d)$ or $\delta(\Delta(\widehat N^n_d))$ vanishes.
The existence of a triangulable cover would provide an example of an action of a finite group of even order $d$ that cannot be made simplicial even up to changing the triangulation.
Notice that in dimension $5$ this happens already for $d=2$: the Galewski-Stern $5$-manifold from \cite{GS79} is non-triangulable and non-orientable, but all orientable closed 5-manifolds are triangulable; see \cite{SI70}.

This problem about virtual vanishing of certain $\zz_2$-cohomology classes is reminiscent of the following question about vector bundles: given a \textit{flat} vector bundle (i.e., a vector bundle with discrete structure group) over a compact polyhedron $B$, is there a finite cover of $B$ on which the bundle becomes trivial? 
Triviality of a bundle over $B$ with discrete structure group contained in $\slr n$ is obstructed by a class in $H^2(\pi_1(B),\zz_2)$, when $n\geq 3$.
Millson constructed in \cite{MI79} examples of flat real bundles for which this class remains non-trivial in all finite covers, i.e., flat  real bundles that are not virtually trivial.
The analogous problem in the complex case is different: Deligne and Sullivan showed in \cite{DS75} that all  flat complex vector bundles are virtually trivial.
\end{remark}


\appendix
\section{A criterion for residual finiteness}\label{sec:appendix}
\begin{center} \textsc{by Daniel Groves and Jason  Manning}\end{center}
\medskip

In this appendix, we explain why groups satisfying the conclusions of Theorem~\hyperref[thm:action main]{A} are residually finite whenever the peripheral subgroups are.
In fact we prove the somewhat stronger statement that such groups are separable on full quasi-convex subgroups.  Our proof relies on various group theoretic \emph{Dehn filling} results so we begin by recalling the relevant definitions.
\begin{definition}
  Let $(G,\mc{P})$ be relatively hyperbolic.  A \emph{Dehn filling} of $(G,\mc{P})$ is a quotient $\overline{G}$ of $G$ so that $\ker(G\to \overline{G})$ is generated by the union of a collection $\mc{N} = \{N_P\mid P\in \mc{P}\}$ where each $N_P\lhd P$.  We may also write $\overline{G}$ as $G(\mc{N})$ if we want to keep track of the kernel.
  The filling is \emph{peripherally finite} if $[P:N_P]<\infty$ for all $P\in \mc{P}$.  If $\mc{H}$ is a family of subgroups of $G$, and we have $N_P<H^g$ whenever $H^g\cap P$ is infinite, then the filling is called an \emph{$\mc{H}$--filling}.
\end{definition}

Most Dehn filling results require the assumption that the filling is ``sufficiently long''.  To be precise, we have the following.
\begin{definition}
  A statement $\mathsf{F}$ \emph{holds for sufficiently long fillings} of $(G,\mc{P})$ if there is a finite set $S\subset G\setminus\{1\}$ so that the statement $\mathsf{F}$ holds for $G(\mc{N})$ whenever $\mc{N}$ satisfies $\bigcup\mc{N}\cap S = \emptyset$.

  If $\mathsf{G}$ is a property of fillings, we say that $\mathsf{F}$ \emph{holds for sufficiently long $\mathsf{G}$ fillings} if ``$\mathsf{F}$ or not $\mathsf{G}$'' holds for sufficiently long fillings.  For example $\mathsf{G}$ could be the property of being peripherally finite, or of being an $\mc{H}$--filling (or both).
\end{definition}
Notice that a conjunction of \emph{finitely many} statements which hold for sufficiently long fillings also holds for sufficiently long fillings.

Recall that a relatively quasiconvex subgroup $H \subseteq G$  is \textit{full} if for every $g\in G$ and $P\in \mathcal P$ the intersection $H^g \cap P$ is either finite or of finite index in $P$.
We will also need the following definition of Osin in the course of the proof.
\begin{definition}\label{def:strong rqc} \cite[Definition 1.8]{OS06}
A relatively quasi-convex subgroup $H \le G$ is \emph{strongly relatively quasi-convex} if for every $g \in G$ and $P \in \mc{P}$ the intersection $H^g \cap P$ is finite.
\end{definition}

\begin{theorem}\label{thm: appendix main}
Suppose that $(G,\mc{P})$ is a relatively hyperbolic pair with each element of $\mc{P}$ residually finite.  Suppose also that $G$ admits an isometric and cubical action on a $\CAT(0)$ cube complex $X$ so that
\begin{enumerate}
\item\label{itm:cocompact} $\leftQ{X}{G}$ is compact;
\item\label{itm:Pelliptic} Each $P \in \mc{P}$ acts elliptically on $X$; and
\item\label{itm:cubestabs} For each cube $\sigma \in X$, $\Stab_G(\sigma)$ is either conjugate to an element of $\mc{P}$, or else is full relatively quasi-convex, hyperbolic, and virtually compact special. 
\end{enumerate}
Then every full relatively quasi-convex subgroup of $G$ is separable.  In particular, $G$ is residually finite.
\end{theorem}

\begin{proof}  
  We first argue that it is enough to prove the theorem replacing \eqref{itm:cubestabs} with the stronger hypothesis:
  \begin{enumerate}[label = (\arabic*')]
    \setcounter{enumi}{2}  
    \item\label{itm:cubestabsprime} Every cube stabilizer is either maximal parabolic or strongly quasi-convex and virtually compact special.
\end{enumerate}
Let $\mc{P}'\subseteq \mc{P}$ be obtained by omitting those elements of $\mc{P}$ which are both hyperbolic and virtually compact special.
  Every full relatively quasi-convex subgroup of $(G,\mc{P}')$ is full relatively quasi-convex in $(G,\mc{P})$, so if we prove the theorem for $(G,\mc{P}')$ we will have proved it also for $(G,\mc{P})$.
  Replacing $\mc{P}$ by $\mc{P}'$ does not change hypotheses~\eqref{itm:cocompact} or~\eqref{itm:Pelliptic}.  Suppose $H$ is a cube stabilizer which is not maximal parabolic.  With respect to $\mc{P}$ it is therefore full relatively quasi-convex, hyperbolic, and virtually compact special.  By \cite[Theorem 9.1]{HR10}, $H$ is hyperbolic relative to a collection $\mc{D}_H$ of finite index subgroups of conjugates of some parabolics $\mc{P}_H\subset\mc{P}$ (which may occur with multiplicity).  Each of these subgroups is undistorted in $H$ \cite[Lemma 5.4]{OS06}.  In particular each $D\in \mc{D}_H$ is hyperbolic and (using \cite{H08}) virtually special.  It follows that each $P\in \mc{P}_H$ is hyperbolic and virtually special, so $\mc{P}_H\subset \mc{P}\setminus \mc{P}'$.  In particular $H$ is strongly relatively quasi-convex with respect to $\mc{P}'$.  Since $H$ was an arbitrary non-maximal-parabolic cube stabilizer the action of $(G,\mc{P}')$ on $X$ satisfies the stronger condition~\ref{itm:cubestabsprime}.

  We therefore suppose all cube stabilizers are maximal parabolic or strongly quasi-convex and virtually compact special.    Let $\mc{H}$ be a collection of conjugacy representatives of stabilizers of cubes in $X$ which are not maximal parabolic.
  
Let $L$ be a full relatively quasi-convex subgroup of $(G,\mc{P})$, and suppose that $g \in G \smallsetminus L$.
  
Let $\sigma_1, \ldots , \sigma_k$ be representatives of $G$--orbits of cubes in $X$.
For each $i$, let $Q_i$ be the finite-index subgroup of $\Stab(\sigma_i)$ which fixes $\sigma_i$ pointwise, and let $\mc{Q} = \{ Q_1, \ldots Q_k \}$, and $\mc{Q}' = \mc{Q} \cup \{ L \}$.  Note that each element of $\mc{Q}'$ is full relatively quasi-convex.

\begin{claim*} A sufficiently long peripherally-finite $\mc{Q}'$--filling $\pi \co G \to \overline{G} = G/K$ will have the following properties:
\begin{enumerate}
\item $\overline{G}$ is hyperbolic;
\item $\leftQ{X}{K}$ is a $\CAT(0)$ cube complex;
\item For each $H \in \mc{H}$ the map $\pi|_H$ is injective on $H$, and the image $\overline{H}$ is quasi-convex in $\overline{G}$;
\item $\pi(g) \not\in \pi(L)$.
\end{enumerate}
\end{claim*}
\begin{proof}[Proof of Claim]
  The fundamental theorem of relatively hyperbolic Dehn filling \cite{OS07} says that for sufficiently long fillings $\overline{G}$ is hyperbolic relative to the images of the elements of $\mc{P}$.  When these images are finite, this implies that $\overline{G}$ is hyperbolic.

The second item follows from \cite[Corollary 6.6]{GM23}.

The third item follows from \cite[Propositions 4.5 and 4.6]{GM21}
(the induced filling of each $H$ has trivial filling kernels, because $H$ is strongly quasi-convex, so does not intersect maximal parabolics except in finite groups, which can be avoided for sufficiently long fillings).

The fourth item follows from \cite[Proposition 4.7]{GM21}.
\end{proof}
That there exist fillings of the sort in the claim follows from the assumption that elements of $\mc{P}$ are residually finite.  Indeed, for each $P\in \mc{P}$ there are finitely many infinite intersections $P\cap Q^g$ for $Q\in \mc{Q}$, $g\in G$; each such $P\cap Q^g$ is finite index in $P$.  Let $\dot{P}$ be the normal core of the intersection of these $P\cap Q^g$.  For any collection of finite index $\{N_P'\lhd P\}$, the collection $\{N_P = N_P'\cap \dot{P}\}$ determines a peripherally finite $\mc{Q}'$--filling.  Since each $P$ is residually finite, we may choose the $N_P'$ to avoid any given finite set $S$.

Now, $\overline{G}$ acts cocompactly on $\overline{X} = \leftQ{X}{K}$, with stabilizers which are either finite or quasi-convex and virtually special, and hence by \cite[Theorem D]{GM23} $\overline{G}$ is virtually special.  The image $\pi(L)$ of $L$ in $\overline{G}$ is quasi-convex, and $\pi(g) \not\in \pi(L)$, so since quasi-convex subgroups of hyperbolic virtually compact special groups are separable by \cite[Theorem 1.3]{HW08}, $\pi(g)$ can be separated from $\pi(L)$ in a finite quotient of $\overline{G}$.  This is a finite quotient of $G$ separating $g$ from $L$.

Since $L$ and $g \in G \smallsetminus L$ were arbitrary, all full relatively quasi-convex subgroups of $G$ are separable.  Since $\{ 1 \}$ is a full relatively quasi-convex subgroup of $G$, $G$ is residually finite.
\end{proof}

\begin{remark}
By using versions of the Malnormal Special Quotient Theorem on the non-parabolic cell stabilizers, one can weaken Hypothesis \eqref{itm:cubestabs} on these subgroups to, for example, be relatively hyperbolic with respect to the peripheral structure induced from $(G,\mc{P})$, and admitting a (weakly) relatively geometric action on a CAT$(0)$ cube complex.
\end{remark}


\printbibliography

\end{document}